\documentclass[12pt]{amsart}
\usepackage{tikz, hyperref, amssymb}
\usetikzlibrary{calc,3d}
\makeatletter
\pgfset{/tikz/cs/radius/.store in=\tikz@cs@radius}
\makeatother

\usepackage{dsfont}

\newtheorem{thm}{Theorem}[section]
\newtheorem{cor}[thm]{Corollary}
\newtheorem{lem}[thm]{Lemma}
\newtheorem{prop}[thm]{Proposition}

\theoremstyle{definition}
\newtheorem{dfn}[thm]{Definition}
\newtheorem{dfns}[thm]{Definitions}
\newtheorem{ntn}[thm]{Notation}

\theoremstyle{remark}
\newtheorem{rmk}[thm]{Remark}
\newtheorem{rmks}[thm]{Remarks}
\newtheorem{example}[thm]{Example}
\newtheorem{examples}[thm]{Examples}

\numberwithin{equation}{section}

\newcommand{\CC}{\mathbb{C}}
\newcommand{\FF}{\mathbb{F}}

\newcommand{\NN}{\mathbb{N}}

\newcommand{\TT}{\mathbb{T}}
\newcommand{\ZZ}{\mathbb{Z}}
\newcommand{\RR}{\mathbb{R}}

\newcommand{\Cc}{\mathcal{C}}

\newcommand{\bI}{\mathbf{I}}

\newcommand{\1}{\mathbf{1}}
\newcommand{\Tor}{{\mathrm{Tor\,}}}
\newcommand{\id}{\operatorname{id}}

\newcommand{\image}{\operatorname{Im}}
\newcommand{\Aut}{\operatorname{Aut}}
\newcommand{\rank}{\operatorname{rank}}
\newcommand{\coker}{\operatorname{coker}}
\newcommand{\Hom}{\operatorname{Hom}}
\newcommand{\Ext}{\operatorname{Ext}}

\newcommand{\dom}{\operatorname{dom}}
\newcommand{\cod}{\operatorname{cod}}
\newcommand{\Top}{\operatorname{top}}
\newcommand{\op}{\operatorname{op}}
\newcommand{\sign}{\operatorname{sign}}

\newcommand{\Bcub}[1]{B^{#1}}
\newcommand{\Ccub}[1]{C^{#1}}
\newcommand{\Hcub}[1]{H^{#1}}
\newcommand{\Zcub}[1]{Z^{#1}}
\newcommand{\dcub}[1]{\delta^{#1}}

\newcommand{\Cube}[1]{\ensuremath{\mathds{1}_{#1}}}

\title[Homology of $k$-graphs]{Homology for higher-rank graphs and twisted $C^*$-algebras}

\author{Alex Kumjian}
\address{Alex Kumjian\\ Department of Mathematics (084)\\ University
of Nevada\\ Reno NV 89557-0084\\ USA} \email{alex@unr.edu}

\author{David Pask}
\address{David Pask, Aidan Sims\\ School of Mathematics and
Applied Statistics  \\
University of Wollongong\\
NSW  2522\\
AUSTRALIA} \email{dpask, asims@uow.edu.au}
\author{Aidan Sims}

\thanks{This research was supported by the ARC.  Part of the work was completed while the first author was employed at the
University of Wollongong on the ARC grant DP0984360.}

\date{7 October 2011}

\begin{document}

\begin{abstract}
We introduce a homology theory for $k$-graphs and explore its fundamental properties. We establish
connections with algebraic topology by showing that the homology of a $k$-graph coincides with the
homology of its topological realisation as described by Kaliszewski et al. We exhibit combinatorial
versions of a number of standard topological constructions, and show that they are compatible, from
a homological point of view, with their topological counterparts. We show how to twist the
$C^*$-algebra of a $k$-graph by a $\mathbb{T}$-valued $2$-cocycle and demonstrate that examples
include all noncommutative tori. In the appendices, we construct a cubical set
$\widetilde{Q}(\Lambda)$ from a $k$-graph $\Lambda$ and demonstrate that the homology and
topological realisation of $\Lambda$ coincide with those of
 $\widetilde{Q}(\Lambda)$ as defined by Grandis.
 \end{abstract}

\maketitle

\section{Introduction}
In this paper we initiate the study of homology for higher-rank graphs. We develop a suite of
fundamental results and techniques, and also establish connections with a number of related areas:
Via the topological realisations of $k$-graphs introduced in \cite{kkqs}, we establish 
connections with the cubical approach to algebraic topology used in \cite{Massey}. We also show in
an appendix how our approach connects the theory of $k$-graphs to the theory of cubical sets
discussed in, for example, \cite{BrownHiggins:CTGD81, FajstrupRaussenetal:tcs2006,
Grandis2005, GrandisMauri:TAC03, Isaacson:JPAA11}. Our key motivation, however, is that our
homology theory and in particular the associated cohomology theory promises to have an interesting
application to $C^*$-algebras. We discuss this application in Section~\ref{sec:cohomology}: we
introduce the cohomology theory corresponding to our homology and show that $\TT$-valued
$2$-cocycles on a $k$-graph can be used to twist its $C^*$-algebra. As examples we obtain all
noncommutative tori and the  Heegaard-type quantum 3-spheres of Baum, Hajac, Matthes and
Szyma\'nski (see \cite{BaumHajacEtAl:K-th05}). A more detailed study of the cohomology of
$k$-graphs and the structure theory of the associated $C^*$-algebras will be the subject of future
work.

Higher-rank graphs, or $k$-graphs, were introduced by the first two authors in \cite{KP2000} as a
combinatorial model for the higher-rank Cuntz-Krieger algebras discovered and analysed by Robertson
and Steger \cite{RobertsonSteger:JRAM99}, and to unify the constructions of many other interesting
$C^*$-algebras \cite{KumjianPask:ETDS99}. The $C^*$-algebras of higher-rank graphs have been
studied by numerous authors over the last decade (see, for example, \cite{DavidsonYang:CJM09,
DavidsonYang:NYJM09, Evans:NYJM08, FarthingMuhlyEtAl:SF05, SkalskiZacharias:HJM08,
SkalskiZacharias:JOT10, Yamashita:xx09}).

The combinatorial properties of a $k$-graph suggest a sort of $k$-dimensional directed graph, and
this point of view has been borne out in numerous ways in the study of $k$-graph $C^*$-algebras.
More recently, however, it has begun to suggest relationships with
topology. These connections first arose in
\cite{PaskQuiggEtAl:NYJM04, PaskQuiggEtAl:JA05} where a theory of coverings and a notion of
fundamental group for $k$-graphs was developed. These notions closely parallel the topological
theory, but were motivated by $C^*$-algebraic considerations: the authors
demonstrated that coverings of $k$-graphs correspond to relative skew products which in turn
correspond to coaction crossed products and crossed products by homogeneous spaces.

The topological flavour of some of the results of~\cite{PaskQuiggEtAl:NYJM04, PaskQuiggEtAl:JA05}
suggest that each $k$-graph should have a topological realisation, which would be a $k$-dimensional
CW complex, and that the $k$-graph could profitably be viewed as a combinatorial version of its
topological realisation \cite[Section~6]{PaskQuiggEtAl:NYJM04}. Current work of the first and third
authors with Kaliszewski and Quigg \cite{kkqs} bears this idea out, showing in particular that the
fundamental groups of a $k$-graph and of its topological realisation are isomorphic and that many
well-known $k$-graph constructions are well-behaved with respect to fundamental groups.

In the current paper, we expand on this idea further by commencing the study of homology of
higher-rank graphs. After recalling basic definitions and notation in Section~\ref{sec:prelims}, we
proceed in Section~\ref{sec:homology} to define our homology, prove that it is a functor, show that
we can measure connectedness by the $0$\textsuperscript{th} homology group, and show that the
$1$-cycles correspond naturally to integer combinations of undirected cycles in the $k$-graph.

In Section~\ref{sec:theorems}, we prove analogs of a number of standard theorems in algebraic
topology for our homology. For example we show that the K\"{u}nneth formula holds for the homology
of a cartesian product of higher-rank graphs, and that the homology of the quotient of an acyclic
$k$-graph by a free action of a discrete group $G$ is isomorphic to the homology of $G$. We also
show that every automorphism of a $k$-graph induces a long exact sequence in homology which
corresponds exactly to the long exact sequence for a mapping torus.

In Section~\ref{sec:examples}, we use a combination of these results and direct calculation to
describe examples of $2$-graphs whose homology is identical to that of the sphere, the torus, the
Klein bottle and the projective plane respectively; we also present these examples in a way which
indicates that their topological realisations should coincide with these four spaces. Details of
these homeomorphisms will appear in \cite{kkqs}. In Section~\ref{sec:geometric}, we use an argument
based on that given by Hatcher for simplicial complexes and singular homology
\cite{Hatcher:AlgTop02}, to show that our homology for a $k$-graph agrees with the singular
homology of its topological realisation. This suggests strongly that our homology theory is a
reasonable one for $k$-graphs.

Section~\ref{sec:cohomology} gives a taste of the $C^*$-algebraic application which motivates our
study of homology for $k$-graphs: twisted $k$-graph $C^*$-algebras. We briefly discuss the
cohomology of a higher-rank graph and check that it satisfies the Universal Coefficient Theorem. We
introduce the notion of the $C^*$-algebra of higher-rank graph twisted by a $\TT$-valued 2-cocycle,
and show that the isomorphism class of the $C^*$-algebra depends only on the cohomology class of
the cocycle. We then consider some basic examples of finite $k$-graphs whose twisted $C^*$-algebras
capture the noncommutative tori and the Heegaard-type quantum 3-spheres of
\cite{BaumHajacEtAl:K-th05}.

Our homology is modeled on the cubical version of singular homology in \cite{Massey} and is closely
related to the homology of a cubical set introduced by Grandis \cite{Grandis2005}. We establish in
Appendix~\ref{app:grandis} that a $k$-graph $\Lambda$ determines a cubical set
$\widetilde{Q}(\Lambda)$, and that our homology of  $\Lambda$  is isomorphic to Grandis' homology
of $\widetilde{Q}(\Lambda)$. Hence, in principle, some of our earlier results
(Theorem~\ref{thm:freeaction} and part of the statement of Theorem~\ref{thm:kunneth}) could be
recovered from Grandis'. However we provide a self-contained treatment avoiding unnecessary
complications involving degeneracy maps: we believe that the resulting simplicity of presentation
justifies our approach. We demonstrate in Appendix~\ref{app:realisations}, that  the topological
realisation of a $k$-graph as described in \cite{kkqs} is homeomorphic to the topological
realisation, outlined in \cite{Grandis2005}, of the associated cubical set.

\subsection*{Acknowledgements.}
The idea that homology of $k$-graphs might be of interest first arose from the study of topological
realizations (see \cite{kkqs, PaskQuiggEtAl:NYJM04}), which was suggested by John Quigg. We thank
Mike Whittaker for a number of helpful discussions and in particular for his contributions to
Examples \ref{ex:Klein}~and~\ref{ex:proj plane}. The second author thanks his coauthors for their
hospitality.

\section{Preliminaries}\label{sec:prelims}

As in \cite{kps2}, in our definition of a $k$-graph we will allow for the possibility of $0$-graphs
with the convention that $\NN^0$ is the trivial semigroup $\{0\}$. We insist that all $k$-graphs
are nonempty.

We adopt the conventions of \cite{kps2, PaskQuiggEtAl:NYJM04} for $k$-graphs. Given a nonnegative
integer $k$, a \emph{$k$-graph} is a nonempty countable small category $\Lambda$ equipped with a
functor $d :\Lambda \to \NN^k$ satisfying the \emph{factorisation property}: for all $\lambda \in
\Lambda$ and $m,n \in \NN^k$ such that $d( \lambda )=m+n$ there exist unique $\mu ,\nu \in \Lambda$
such that $d(\mu)=m$, $d(\nu)=n$, and $\lambda=\mu \nu$. When $d(\lambda )=n$ we say $\lambda$ has
\emph{degree} $n$. We often use the same symbol $d$ to denote the degree functor in all $k$-graphs in this paper.

For $k \ge 1$, the standard generators of $\NN^k$ are denoted $e_1, \dots, e_k$, and for $n \in
\NN^k$ and $1 \le i \le k$ we write $n_i$ for the $i^{\rm th}$ coordinate of $n$. For $n = (n_1 ,
\ldots , n_k ) \in \NN^k$ let $| n | = \sum_{i=1}^k n_i$. If $\Lambda$ is a $k$-graph, then for
$\lambda \in \Lambda$, we write $|\lambda|$ for $|d(\lambda)|$. For $m,n \in \NN^k$ we write $m \le n$ if $m_i \le n_i$ for all $i \le k$.
We often implicitly identify $\NN^{k_1 + k_2} =\NN^{k_1} \times \NN^{k_2}$.

Given a $k$-graph $\Lambda$ and $n \in \NN^k$, we write $\Lambda^n$ for $d^{-1}(n)$.
The \emph{vertices} of $\Lambda$ are the
elements of $\Lambda^0$. The factorisation property implies that $o \mapsto \id_o$ is a bijection
from the objects of $\Lambda$ to $\Lambda^0$. The domain and codomain maps in the category
$\Lambda$ therefore determine maps $s,r : \Lambda \to \Lambda^0$: for $\alpha \in\Lambda$, the
\emph{source} $s(\alpha)$ of $\alpha$ is the identity morphism associated with the object
$\dom(\alpha)$ and similarly, $r(\alpha) = \id_{\cod(\alpha)}$. An \textit{edge} in a $k$-graph
is a morphism $f$ with $d(f) = e_i$ for some $i=1, \ldots , k$.  In keeping with graph terminology an
 element $\lambda \in \Lambda$ is often called a \emph{path}.

A $0$-graph is then a countable category whose only morphisms are the identity morphisms, which we
regard as a collection of isolated vertices.

Each $1$-graph $\Lambda$ is the path-category of the directed graph with vertices $\Lambda^0$ and
edges $\Lambda^1$ and range and source maps inherited from $\Lambda$. Conversely, if $E$ is a
directed graph, then its path-category $E^*$ is a $1$-graph under the length function.
This leads to the unusual convention that a path in $E$ is a sequence of edges $\alpha_1 \cdots
\alpha_n$ such that $s(\alpha_i) = r(\alpha_{i+1})$ for all $i$, and we write $r(\alpha) =
r(\alpha_1)$ and $s(\alpha) = s(\alpha_n)$.

Let $\lambda$ be an element of a $k$-graph $\Lambda$ and suppose that $0 \le m \le n \le d(\lambda)$. By the factorisation
property there exist unique elements $\alpha \in \Lambda^m$, $\beta \in \Lambda^{n-m}$ and $\gamma \in
\Lambda^{d(\lambda) - n}$ such that $\lambda = \alpha\beta\gamma$. We define $\lambda(m,n) := \beta$. We then have
$\lambda(0,m) = \alpha$ and $\lambda(n, d(\lambda)) = \gamma$. In particular, for $0 \le m \le d(\lambda)$,
\[
\lambda = \lambda(0,m)\lambda(m, d(\lambda)).
\]

For $v\in\Lambda^0$ and $E \subset \Lambda$, we write $v E$ for $E \cap r^{-1}(v)$ and $E v$ for $E
\cap s^{-1}(v)$.

\begin{dfn}[{\cite[Definition~5.1]{KP2000} (see also \cite{PaskQuiggEtAl:JA05})}] \label{def:spg}
Let $G$ be a discrete group, $( \Lambda , d)$ a  $k$-graph and  $c : \Lambda \rightarrow G$ a
functor. The \emph{skew product} $k$-graph $\Lambda \times_c G$ is defined as follows:
as a set $\Lambda \times_c G$ is the cartesian product $\Lambda \times G$ and
$d ( \lambda , g ) = d ( \lambda )$ (so $(\Lambda \times_c G)^0 = \Lambda^0 \times G$) with
\[
s (\lambda , g ) = (  s ( \lambda ) ,  g c ( \lambda ) ) \quad\mbox{and} \quad
r ( \lambda , g ) = ( r ( \lambda ), g).
\]
If $s ( \lambda ) = r ( \mu )$ then $( \lambda , g )$ and $( \mu , g c ( \lambda ) )$ are
composable in $\Lambda \times_c G $ and
\begin{equation} \label{eq:spgcomp}
    ( \lambda , g ) (  \mu , g c ( \lambda ) ) = ( \lambda \mu , g ) .
\end{equation}
\end{dfn}

\begin{examples} \label{ex:kgraphs}
\begin{enumerate}
\item\label{it:Tsubk} For $k \ge 0$ let $T_k = \NN^k$ regarded as a $k$-graph with $d : T_k \to
    \NN^k$ the identity map.  So $T_k$ has exactly one morphism of degree $n$ for each $n \in
    \NN^k$, and in particular a single vertex $0$. For $k \ge 1$, $T_k$ is generated by the $k$
    commuting elements, $e_1, \dots, e_k$.
\item\label{it:Bsubn}  For $n \ge 1$ let $B_n$ be the path category of the directed graph with one vertex and
    $n$ distinct edges $f_1 , \dots , f_n$. We refer to $B_n$ as the $1$-graph associated
    to the bouquet of $n$-circles (see Example~\ref{ex:fga}(\ref{it:Cayley})).
\item For $n \ge 2$ let  $\mathbb{F}_n$ be the free group on $n$ generators $\{ h_1 , \ldots , h_n \}$
and define the functor $c : B_n \to \mathbb{F}_n$  by $c( f_i ) = h_i$ for $i=1, \ldots , n$.
Let $A_n$ denote the skew product $1$-graph $B_n \times_c \mathbb{F}_n$.
The underlying directed graph associated to $A_n$ is the (right) Cayley graph of
    $\mathbb{F}_n$ and may be visualised as a uniform $n$-ary tree.
\item For $k \ge 1$ and $m \in (\NN \cup \{\infty\})^k$, we write $\Omega_{k,m}$ for the
    $k$-graph with
    \[
        \Omega_{k,m} := \{(p,q) \in \NN^k \times \NN^k : p \le q \le m\}
    \]
    and with structure maps $r(p,q) := (p,p)$, $s(p,q) := (q,q)$, $d(p,q) := q-p$ and $(p,q)(q,r)
    := (p,r)$. Define $\Omega_0
    := \{0\}$ and for $k \ge 1$ let $\Omega_{k} := \Omega_{k,(\infty, \dots, \infty)}$.
\item For $k \ge 1$, let $\Delta_k$ be the $k$-graph with $\Delta_k := \{(p, q) \in \ZZ^k \times
    \ZZ^k : p \le q\}$ and structure maps as in $\Omega_{k,m}$.
\item Let  $( \Lambda_i , d_i)$ be a $k$-graph for $i=1,2$. The disjoint union $\Lambda_1 \sqcup
    \Lambda_2$ may be regarded as a $k$-graph with $d(\lambda) = d_i(\lambda)$ if
    $\lambda \in \Lambda_i$ and  with other structure maps likewise inherited from the $\Lambda_i$.
\item Let $( \Lambda_i , d_i )$ be a $k_i$-graph for $i=1,2$. Then $( \Lambda_1 \times
    \Lambda_2, d_1 \times d_2 )$ is a $(k_1 + k_2)$-graph where $\Lambda_1 \times \Lambda_2$ is
    the product category and $d_1 \times d_2 : \Lambda_1 \times \Lambda_2 \rightarrow \NN^{k_1
    + k_2}$ is given by $(d_1 \times d_2) (\lambda_1 , \lambda_2 ) = ( d_1 ( \lambda_1 ) , d_2
    ( \lambda_2 ) ) \in \NN^{k_1} \times \NN^{k_2}$ for $\lambda_1 \in \Lambda_1$ and
    $\lambda_2 \in \Lambda_2$.
\end{enumerate}
\end{examples}

Let $k_1 , k_2 \ge 1$.
Let $\pi_1 : \ZZ^{k_1+k_2} \to \ZZ^{k_1}$ denote the projection onto the first $k_1$ coordinates
and $\pi_2 : \ZZ^{k_1+k_2} \to \ZZ^{k_2}$ denote the projection onto the last $k_2$ coordinates. We
frequently regard $\pi_i$ as a homomorphism from $\NN^{k_1+k_2}$ to $\NN^{k_i}$.

A \emph{$k$-graph morphism} between $k$-graphs is a degree-preserving functor. There is a
category whose objects are $k$-graphs and whose morphisms are $k$-graph morphisms. Whenever we
regard $k$-graphs as objects of a category in this paper, it will be this one.

\begin{examples} \label{ex:prodex}
\begin{enumerate}
\item For $k_1 , k_2 \ge 1$ we have $T_{k_1+k_2} = \NN^{k_1 + k_2} = \NN^{k_1} \times
    \NN^{k_2} =  T_{k_1} \times T_{k_2}$.
\item For $k_1 , k_2 \ge 1$ we have $\Delta_{k_1+k_2} \cong \Delta_{k_1}
\times \Delta_{k_2}$. One checks that the map
$(m,n) \mapsto ( ( \pi_1 (m) , \pi_1 (n) ) , ( \pi_2 (m) , \pi_2 (n) ) )$ gives the desired isomorphism
of $k$-graphs.
\end{enumerate}
\end{examples}

It is sometimes useful to consider morphisms between higher-rank graphs which do not preserve
degree. The following definition is from \cite[\S 2]{kps2}.

\begin{dfn}
Let $(\Lambda , d )$ be a $k$-graph and $( \Gamma , d' )$ be an $\ell$-graph. A functor $\psi :
\Lambda \to \Gamma$ is called a \emph{quasimorphism} if there is a homomorphism $\pi : \NN^k \to
\NN^\ell$ such that for all $\lambda \in \Lambda$ we have $\pi ( d ( \lambda ) ) = d' ( \psi (
\lambda ) )$.
\end{dfn}

\begin{example} \label{ex:natural_qm}
For $i = 1,2$, let $( \Lambda_i , d_i )$ be a $k_i$-graph. Let $\Lambda_1 \times \Lambda_2$ the
associated cartesian product $(k_1+k_2)$-graph. Since every element $\lambda  \in \Lambda_1 \times
\Lambda_2$ is of the form $\lambda = ( \lambda_1 , \lambda_2 )$ where $\lambda_1 \in \Lambda_1$ and
$\lambda_2 \in \Lambda_2$, for $i=1,2$ there is a natural functor $\psi_i : \Lambda_1 \times
\Lambda_2 \to \Lambda_i$ given by $( \lambda_1 , \lambda_2 ) \mapsto \lambda_i$; 
note that $\psi_i $ is a quasimorphism with $d_i \circ \psi_i = \pi_i \circ (d_1 \times d_2)$.
\end{example}

\begin{dfn} \label{def:pbg}
Let $f : \NN^k \to \NN^l$ be a homomorphism and let $\Gamma$ be an $l$-graph. The pullback
$f^*\Gamma$ is the $k$-graph $\{(\gamma,n) \in \Gamma \times \NN^k : f(n) = d(\gamma)\}$ with degree map 
$d(\gamma,n)= n$ (see \cite[Definition~1.9]{KP2000}). The structure maps
are given by $r ( \gamma , n ) = ( r ( \gamma ) , 0 )$ and $s ( \gamma , n ) = ( s ( \gamma ) , 0
)$. If $s ( \lambda ) = r ( \mu )$ in $\Gamma$ then $( \lambda , n )$ and $( \mu , m)$ are
composable in $f^* \Gamma$, and
\begin{equation} \label{eq:pbgcomp}
( \lambda , n ) ( \mu , m ) = ( \lambda \mu , m+n ) .
\end{equation}
\end{dfn}

All of the above is standard notation for $k$-graphs. In the remainder of this section we introduce
some new notation related to $k$-graphs as a preliminary to the definition and basic properties of
homology for $k$-graphs in Section~\ref{sec:homology}

\begin{dfn} \label{def:undirected}
Let $\Lambda$ be a $k$-graph where $k \ge 1$. For $\lambda \in \Lambda$ and $m \in \{1, -1\}$, we
define
\[
    s(\lambda, m) := \begin{cases}
            s(\lambda) &\text{ if $m = 1$} \\
            r(\lambda) &\text{ if $m = -1$}
    \end{cases}
    \qquad\text{ and }\qquad
    r(\lambda, m) := s(\lambda, -m).
\]
An \textit{undirected path} is a pair $(g,m)$ where $g = (g_1 , \ldots , g_n)$ is a sequence of
edges in $\Lambda$ and $m=( m_1 , \ldots , m_n)$ is a sequence of orientations, $m_i  \in \{1,
-1\}$ such that $s(g_i, m_i) = r(g_{i+1}, m_{i+1})$ for all $i$. If $(g,m)$ is an undirected path,
we define $s(g,m) := s(g_n, m_n)$ and $r(g, m) := r(g_1, m_1)$. If $r(g,m) = s(g,m)$, then we say
that the undirected path $(g,m)$ is \emph{closed}.

A closed undirected path $(g,m)$ is called \emph{simple} if $s(g_i, m_i) \not= s(g_j, m_j)$ for $i
\not= j$.
\end{dfn}

\begin{dfn}{\bf (cf.\ \cite[\S 3]{PaskQuiggEtAl:NYJM04})} \label{dfn:conn}
A $k$-graph $\Lambda$ is \textit{connected} if the equivalence relation on $\Lambda^0$ generated
by $\{(r(\lambda), s(\lambda)) : \lambda \in \Lambda\}$ is $\Lambda^0 \times \Lambda^0$.
\end{dfn}

\begin{rmk}\label{rmk:component decomposition}
A $k$-graph $\Lambda$ is connected if and only if for all $u, v \in \Lambda^0$ there is an
undirected path with source $u$ and range $v$.

For each equivalence class $X \subseteq \Lambda^0$  from
Definition~\ref{dfn:conn}, the $k$-graph $X \Lambda X$ is a connected component of $\Lambda$.
Each $k$-graph is the disjoint union of its connected components.
\end{rmk}

For $k \ge 0$ define $\mathbf{1}_k := \sum^k_{i=1} e_i \in \NN^k$. By convention $\mathbf{1}_0 = 0
\in \NN^0$.

\begin{dfn}
Let $\Lambda$ be a $k$-graph. For $r \ge 0$ let
\[
Q_r ( \Lambda ) = \{ \lambda \in \Lambda : d ( \lambda ) \le \mathbf{1}_k , |\lambda| = r \}.
\]
Let $Q ( \Lambda ) = \cup_{r\ge 0} Q_r ( \Lambda )$.
\end{dfn}

We have $Q_0(\Lambda) = \Lambda^0$, and $Q_r(\Lambda) = \emptyset$ if $r > k$. Let  $0 < r \le k$. The set
$Q_r(\Lambda)$ consists of the morphisms in $\Lambda$ which may be expressed as the composition of a sequence of $r$ edges
with distinct degrees. We regard elements of $Q_r(\Lambda)$ as unit $r$-cubes in the sense that each one gives rise to a
commuting diagram of edges in $\Lambda$ shaped like an $r$-cube. In particular, when $r \ge 1$, each element of
$Q_r(\Lambda)$ has $2r$ faces in $Q_{r-1}(\Lambda)$ defined as follows.

\begin{dfn}
Fix $\lambda \in Q_r(\Lambda)$ and write $d( \lambda ) = e_{i_1} + \cdots + e_{i_r}$ where $i_1 < \cdots < i_r$. For $1
\le j \le r$, define $F_j^0(\lambda)$ and $F_j^1(\lambda)$ to be the unique elements of $Q_{r-1} ( \Lambda )$ such that
there exist $\alpha,\beta \in \Lambda^{e_{i_j}}$ satisfying
\[
F^0_j(\lambda)\beta = \lambda = \alpha F^1_j(\lambda).
\]
\end{dfn}

\begin{rmk}\label{rmk:2faces}
Equivalently, $F_j^0 ( \lambda ) = \lambda(0 , d(\lambda) - e_{i_j})$ and $F_j^1(\lambda) = \lambda (e_{i_j} ,
d(\lambda))$. If $1 \le i < j \le r$, then $F^\ell_i \circ F^m_j = F^m_{j-1}\circ F^\ell_i$ for $\ell , m \in \{0, 1\}$.
\end{rmk}

\begin{ntn}\label{ntn:F.A.G.}
Let $X$ be a set. We write $\ZZ X$ for the free abelian group generated by $X$ (so $\ZZ \emptyset =
\{0\}$).
\end{ntn}

\begin{rmk}\label{rmk:F.A.G. maps}
Let $X$ and $Y$ be sets. Then every function $f : X \to Y$ extends uniquely to a homomorphism $f :
\ZZ X \to \ZZ Y$. In particular, the inclusion maps induce an isomorphism  $\ZZ(X \sqcup Y) \cong
\ZZ X \oplus \ZZ Y$. Moreover there is an isomorphism $\ZZ(X \times Y) \cong \ZZ X \otimes \ZZ Y$
determined by $(x, y) \mapsto x \otimes y$.
\end{rmk}

\section{The homology of a \texorpdfstring{$k$}{k}-graph}\label{sec:homology}

In this section we define the homology of a $k$-graph, compute some basic examples and provide descriptions
of the first two homology groups. Throughout this paper, we use $r$ (for rank) for the indexing subscript in
complexes and in homology groups because $n$ is more commonly used for a generic element of $\NN^k$.

\begin{dfns}
For $r \in \NN$ let $C_r(\Lambda) = \ZZ Q_r(\Lambda)$. For $r \ge 1$, define $\partial_r : C_r(\Lambda) \to C_{r-1} (
\Lambda )$ to be the unique homomorphism such that
\begin{equation} \label{eq:bondarydef}
\partial_r( \lambda ) = \sum_{\ell=0}^1 \sum_{i=1}^r (-1)^{i+\ell} F_i^\ell ( \lambda ) \quad\text{ for all $\lambda \in Q_r(\Lambda)$.}
\end{equation}
We write $\partial_0$ for the zero homomorphism $C_0 ( \Lambda ) \to \{0\}$.
\end{dfns}

\begin{rmks}
For $f \in Q_1 ( \Lambda )$ we have $F_1^1 (f) = s(f)$ and $F_1^0 (f) = r(f)$ and so
$\partial_1 (f) = s(f)-r(f)$.

Fix $\lambda \in Q_2 ( \Lambda )$. Write $d(\lambda) = e_{j_1} + e_{j_2}$ with $j_1 < j_2$.
Factorise $\lambda = f_1 g_1 = g_2 f_2$ where $d(f_i) = e_{j_1}$ and $d(g_i) = e_{j_2}$ for
$i=1,2$. Then $F_2^0 ( \lambda ) = \lambda ( 0 , e_{j_1} )=f_1$, $F_2^1 ( \lambda ) = \lambda (
e_{j_2} , e_{j_1}+e_{j_2})=f_2$, $F_1^0 ( \lambda ) = \lambda ( 0 , e_{j_2} )=g_2$ and $F_1^1 (
\lambda ) = \lambda ( e_{j_1} , e_{j_1} + e_{j_2} )=g_1$. Hence
\begin{equation} \label{eq:partial2}
\partial_2 ( \lambda ) = g_1 + f_1 - f_2 - g_2.
\end{equation}
\end{rmks}

For $r \ge 0$, $\partial_r$ is a homomorphism and $\partial_r \circ
\partial_{r+1} = 0$ by Remark~\ref{rmk:2faces}. Hence we have the following.
\begin{lem}
Let $\Lambda$ be a $k$-graph, then $( C_* ( \Lambda ) , \partial_* )$ is a chain complex.
\end{lem}

We define the homology of $\Lambda$ to be the homology of the chain complex $C_*(\Lambda)$.

\begin{dfn}
For $r \in \NN$ define $H_r ( \Lambda ) = \ker ( \partial_r ) / \image\partial_{r+1}$. We call
$H_r(\Lambda)$ the \emph{$r$\textsuperscript{th} homology group of $\Lambda$} and we call
$H_*(\Lambda)$ the \emph{homology of $\Lambda$}.
\end{dfn}

\begin{lem}
Fix $n \in \NN$. If $\psi : \Lambda_1 \to \Lambda_2$ is a $k$-graph morphism, then there is a homomorphism $\psi_* :
H_r(\Lambda_1) \to H_r(\Lambda_2)$ determined by $\psi_*([\lambda]) = [\psi(\lambda)]$ for all $\lambda \in
Q_r(\Lambda)$. Moreover, the assignments $\Lambda \mapsto H_r (\Lambda)$ and $\psi \mapsto \psi_*$ comprise a
covariant functor from the category of $k$-graphs with $k$-graph morphisms to the category of abelian groups with
homomorphisms.
\end{lem}
\begin{proof}
For $\lambda \in Q_r( \Lambda_1 )$ we have $\psi(\lambda) \in Q_r( \Lambda_2 )$ as $\psi$ is degree
preserving. Since it preserves factorisations, $\psi$ intertwines the face maps on $Q_r(\Lambda_1)$ and
$Q_r(\Lambda_2)$, so it intertwines the boundary maps $\partial_r$ and therefore defines a homomorphism
$\psi_* : H_r( \Lambda_1 ) \to H_r( \Lambda_2 )$.

For the second assertion of the Lemma, we just have to check that $\psi \mapsto \psi_*$ preserves
composition. This follows immediately from the definition.
\end{proof}

\begin{rmk} \label{rmk:trivial}
For a $k$-graph $\Lambda$ and $r>k$, we have $Q_{r} ( \Lambda ) = \emptyset$, so $C_r(\Lambda)$
and $H_r(\Lambda)$ are trivial.
\end{rmk}

\begin{rmk} \label{rmk:disjoint}
Let $\Lambda_i $ be $k$-graphs for $i=1,2$.  Then the chain complex $C_*(\Lambda_1 \sqcup
\Lambda_2)$ decomposes as the direct sum of the complexes $C_*(\Lambda_1)$ and $C_*(\Lambda_2)$.
Thus the canonical inclusions of $\Lambda_1, \Lambda_2$ into $\Lambda_1 \sqcup \Lambda_2$ induce an
isomorphism $H_* ( \Lambda_1 ) \oplus H_* ( \Lambda_2 ) \cong H_* (\Lambda_1 \sqcup \Lambda_2)$.
Indeed, this isomorphism holds for countable disjoint unions of $k$-graphs.
\end{rmk}

\begin{rmk} \label{rmk:opposite}
Let $\Lambda$ be a $k$-graph and let $\Lambda^{\op}$ be the opposite category, which is a $k$-graph
under the same degree map. We write $\lambda^{\op}$ for an element $\lambda \in \Lambda$ when
regarded as an element of $\Lambda^{\op}$. For each $r$, the assignment $\lambda \mapsto (-1)^r
\lambda^{\op}$ induces an isomorphism $\phi_r : C_r(\Lambda) \to C_r(\Lambda^{\op})$. Using that
$F_i^l(\lambda^{\op}) = F_i^{1-l}(\lambda)^{\op}$ for all $\lambda \in Q_r(\Lambda)$, a calculation
shows that $\partial_{r+1} \circ \phi_{r+1} = \phi_r \circ \partial_{r+1}$ for all $r$. So $\phi_*$ is
an isomorphism of complexes and hence induces an isomorphism $H_*(\Lambda) \cong
H_*(\Lambda^{\op})$.
\end{rmk}

\begin{examples} \label{ex:donebydef}
\begin{enumerate}
\item Let $T_0$ be the $0$-graph of  Examples \ref{ex:kgraphs} (\ref{it:Tsubk}). Then $Q_0 ( T_0
    ) = \{ 0 \}$ and $Q_r ( T_0 ) = \emptyset$ for all $r \ge 1$. Hence $C_0 ( T_0 ) = \ZZ \{ 0
    \}$ and $C_r ( T_0 ) = \{0 \}$ for all $r \ge 1$. Since $\partial_r = 0$ for all $r \ge 0$, we
    have $H_0 (T_0 ) = \ZZ\{ 0 \} \cong \ZZ$ and $H_r(T_0) = \{ 0 \}$ for $r \ge 1$.
\item  More generally, for $k \ge 1$, we have $Q_0(T_k) = \{ 0 \}$, $Q_r(T_k) =  \emptyset$ for
    all $r > k$ and
\[
Q_r(T_k) = \{ e_{i_1} + \cdots +  e_{i_r} \mid 1 \le i_1 < \cdots < i_r \le k \}
\]
for $1 \le r \le k$.   Thus $|Q_r(T_k)| = \binom{k}{r}$ for $0 \le r \le k$.

For $1 \le j \le r \le k$, we have $F_j^0 = F_j^1$, so  $\partial_r = 0$. Hence
\[
H_r(T_k) = \ZZ Q_r(T_k) \cong \ZZ^{\binom{k}{r}}
\qquad\textrm{ for $0 \le r \le k$,}
\]
and $H_r(T_k) = \{ 0 \}$ for $r > k$.  In particular $T_k$ has the same homology as the $k$-torus
$\TT^k$.
\end{enumerate}
\end{examples}

\begin{dfn}
Let $\Lambda$ be a $k$-graph and let $(g,m)$ be an undirected path in $\Lambda$ (see Definition
\ref{def:undirected}). Then
\[
h = \sum_{i=1}^{n} m_{i}  g_{i}  \in C_1 ( \Lambda )
\]
is called the \emph{trail} associated to $(g,m)$. If $(g,m)$ is closed, then $h$ is said to be a
\emph{closed trail}. If in addition $(g,m)$ is simple, then $h$ is called a \emph{simple closed trail}.
\end{dfn}

\begin{rmk} \label{rmk:boundaryofudp}
Let $(g,m)$ be an undirected path  in $\Lambda$  with source $u$ and range $v$. A straightforward
computation shows that $\partial_1 (h) = u - v$ where $h$ is the trail associated to $(g,m)$. Hence, if $h$
is a closed trail then $\partial_1 (h) = 0$. If $h$ is a closed trail and $a \in \ZZ$ is nonzero then $ah$ is
also a closed trail.
\end{rmk}

\begin{prop} \label{prop:H0ofconn}
Let $\Lambda$ be a connected $k$-graph, then $H_{0} ( \Lambda ) \cong \ZZ$.
\end{prop}
\begin{proof}
Define a homomorphism $\theta : C_{0} ( \Lambda ) \to \ZZ$ by $\theta ( v ) =1$ for all $v \in
\Lambda^{0}$. It suffices to show that $\ker (\theta ) \subset \image ( \partial_{1} )$, as the
reverse inclusion is clear.

Fix distinct $u, v \in \Lambda^{0}$. Since $\Lambda$ is connected there is an undirected path
$(g,m)$ from $u$ to $v$. By Remark \ref{rmk:boundaryofudp} $\partial_{1} (h)= u - v$ where $h$ is
the trail associated to $(g,m)$. In particular, $u - v \in \image(\partial_1)$.

Let $a = \sum_{i=1}^{n} m_{i}  v_{i}  \in \ker ( \theta )$ with distinct $v_{i}$ and $m_{i} \neq 0$
for all $i$. We prove by induction on $n \ge 2$ that $\sum_{i=1}^{n} m_{i} v_{i}  \in \image(
\partial_{1} )$. When $n=2$ we must have $m_1 + m_2 = 0$. The preceding paragraph yields a trail $h$
such that $\partial_1(h) = v_1 - v_2$, and then $a = \partial_1(m_1 h) \in \image(\partial_{1})$.

Fix $n \ge 3$ and suppose  the result holds  for all $\ell$ with $n > \ell \ge 2$. Relabeling if
necessary, we may assume that $m_{1}$ and $m_{2}$ have opposite sign, and $| m_{1} | \le |  m_{2}
|$. We give a proof for the case $m_{1} > 0$,  the case $m_{1} <0$ being similar. Since $\Lambda$
is connected there is an undirected path $(g_1,m_1)$ from $v_1$ to $v_2$. Let $h_{1} \in C_{1} (
\Lambda )$ be the associated trail.  Then
 $\partial_{1} (h_{1}) =  v_{1} -  v_{2}$ and
\[
a_{1} = a - \partial_{1}(  m_{1} h_{1} ) =  (m_{2} + m_{1} )  v_{2} +
\sum_{i=3}^{n} m_{i}  v_{i}.
\]
By the inductive hypothesis $a_{1} \in \image ( \partial_{1} )$ and so $a = a_1 + \partial ( m_1
g_1 ) \in \image ( \partial_{1} )$.
\end{proof}

Combining Proposition~\ref{prop:H0ofconn}, Remark~\ref{rmk:disjoint} and Remark~\ref{rmk:component
decomposition} gives the following.

\begin{cor} \label{cor:conncomp}
Let $\Lambda$ be a $k$-graph with $p$ connected components (where $p \in \{1, 2, \dots\} \cup
\{\infty\}$). Then $H_{0} ( \Lambda ) \cong \ZZ^{p}$. In particular $\Lambda$ is connected if and
only if $H_0 ( \Lambda ) \cong \ZZ$.
\end{cor}

\begin{example} \label{ex:delta1}
Since $\Delta_1$ is connected we have $H_0 ( \Delta_1 ) \cong \ZZ$ by
Proposition~\ref{prop:H0ofconn}. We claim that  $H_r( \Delta_1 ) = 0$ for all $r \ge 1$. By
Remark~\ref{rmk:trivial} it suffices to check that $H_1(\Delta_1) = \{0\}$. To see this fix $f \in
C_1(\Delta_1) \setminus \{0\}$. Then we may express $f = \sum_{i=\ell}^m a_i (i, i+1)$, where $a_i
\in \ZZ$ and $a_m \ne 0$. Then
\begin{align*}
\partial_1 ( f ) &= \sum_{i=\ell}^m a_i \left((i+1, i+1) - (i, i) \right) \\
&=  a_m(m+1, m+1) - a_\ell(\ell, \ell) + \sum_{i=\ell + 1}^m (a_{i-1} - a_i)(i, i).
\end{align*}
Since $a_m \neq 0$ it follows that $\partial_1 (f) \neq 0$. So $\partial_1$ is injective and hence
$H_1(\Delta_1) = \ker(\partial_1)$ is trivial.
\end{example}

\begin{prop} \label{prop:trailsforker}
Let $\Lambda$ be a $k$-graph. For each $a \in \ker \partial_1$, there
exist simple closed trails $h_1 , \ldots , h_n$ in $C_1 ( \Lambda )$ such that
$a = \sum_{i=1}^n m_i h_i $.
\end{prop}
\begin{proof}
For $a = \sum_{i=1}^n a_i f_i\in \ker \partial_1$ where the $f_i$ are distinct elements of
$Q_1(\Lambda)$, set $N(a) := \sum^n_{i=1} |a_i|$.  We proceed by induction on $N(a)$. If $N(a) =
0$, the result is trivial. Fix $N > 0$ and suppose as an inductive hypothesis that whenever $N(a) <
N$, there are simple closed trails $h_i$ and integers $m_i$ such that $a = \sum_{i=1}^n m_i h_i$.
Fix $a$ with $N(a) = N$. It suffices to show that there is a simple closed trail $h \in
C_1(\Lambda)$ such that $N(a - h) < N(a)$.

Recall from Definition~\ref{def:undirected} that if $p \in \{1, -1\}$ and $f \in Q_1(\Lambda)$,
then $s(f,p)$ means $s(f)$ if $p = 1$ and $r(f)$ if $p = -1$; and $r(f,p) = s(f,-p)$.

Express $a = \sum_{i=1}^n a_i f_i$ where the $f_i$ are distinct elements of $Q_1(\Lambda)$, and
each $a_i \not= 0$. Let $i_1 := 1$, let $p_1 := \sign(a_1)$. If $s(f_1) = r(f_1)$, then $h := p_1
f_1$ has the desired property. Otherwise, let $v_0 = r(f_1, p_1)$ and $v_1 = s(f_1, p_1)$. Since
the coefficient of $v_1$ in $\partial_1(a)$ is zero, there must exist $i_2$ such that the
coefficient of $v_1$ in $\partial_1(a_{i_2}f_{i_2})$ is nonzero with the opposite sign to that in $\partial_1(p_{1} f_{i_1})$; let $p_2
:= \sign(a_{i_2})$ and let $v_2 = s(f_{i_2}, p_2)$. Observe that $r(f_{i_2}, p_2) = s(f_{i_1},
p_1)$. We may continue iteratively, as long as the $v_i$ are all distinct, to chose an index $i_j$
such that $p_j := \sign(a_{i_j})$ has the property that the coefficient of $v_{j-1}$ in
$\partial_1(p_j f_{i_j})$ has the opposite sign to that in $\partial_1(p_{j-1} f_{i_{j-1}})$ for
each $j$. We then set $v_j := s(f_{i_j}, p_j)$, and observe that $r(f_{i_j}, p_j) = v_{j-1}$. Since
there are only finitely many nonzero coefficients in $a$, this process must terminate: there is a
first $l$ such that $v_l \in \{v_0, v_1, \dots, v_{l-1}\}$; say $v_l = v_q$ where $q < l$. Then $h
:= \sum^l_{j = q+1} p_j f_{i_j}$ is a simple closed trail. Since $p_j=\sign(a_{i_j})$ for each $j$,
we have $N(a - h) = N(a) - (l - q) < N(a)$ as required.
\end{proof}

\section{Fundamental results}\label{sec:theorems}

In this section we prove versions of a number of standard results in homology theory which suggest
that our notion of homology for $k$-graphs is a reasonable one. In Appendix~\ref{app:grandis}, we will
show that each $k$-graph determines in a fairly natural way a cubical set, and that our homology
then agrees with that of Grandis \cite{Grandis2005}. So a number of results in this section could
be recovered from Grandis' work. However, it seems worthwhile to present self-contained proofs
which are consistent with the notation and conventions associated with $k$-graphs.

We begin with a version of the  K\"{u}nneth formula for our homology (see
Theorem~\ref{thm:kunneth}). In order to do this we must show how our chain complexes behave with
respect to cartesian product of $k$-graphs.

Recall from Example \ref{ex:natural_qm} that given a cartesian product graph $\Lambda_1 \times
\Lambda_2$ there are quasimorphisms $\psi_i : \Lambda_1 \times \Lambda_2 \to \Lambda_i$ consistent
with the projections $\pi_i : \NN^{k_1 + k_2} \to \NN^{k_i}$.

\begin{lem}\label{lem:cube iso}
Let $(\Lambda_i,d_i)$ be a $k_i$-graph for $i = 1,2$ and $\Lambda_1 \times
\Lambda_2$ the associated cartesian product $(k_1+k_2)$-graph. Then for $r \ge
0$, we have $Q_r(\Lambda) =
\bigsqcup_{r_1+r_2=r} Q_{r_1} ( \Lambda_1 ) \times Q_{r_2} (\Lambda_2)$. Hence there is an
isomorphism
\begin{equation}\label{eq:cube bijection}
    \Psi_r : C_r (\Lambda_1 \times \Lambda_2) \cong \bigoplus_{r_1 + r_2 = r} C_{r_1}(\Lambda_1) \otimes C_{r_2}(\Lambda_2)
\end{equation}
given by $\Psi_r(\lambda_1, \lambda_2) = \lambda_1 \otimes \lambda_2$.
\end{lem}
\begin{proof}
For the first assertion, just note that $(d_1 \times d_2)(\lambda_1, \lambda_2) \le \1_{k_1 + k_2}$
if and only if $d_i(\lambda_i) \le \1_{k_i}$ for $i = 1,2$. So
\begin{align*}
Q_r(\Lambda_1 \times \Lambda_2)
    &= \{(\lambda_1, \lambda_2): (d_1 \times d_2)(\lambda_1, \lambda_2) \le \1_{k_1 + k_2},
        |\lambda_1| + |\lambda_2| = r\} \\
    &= \bigsqcup_{r_1 + r_2 = r} \{(\lambda_1, \lambda_2):
    d_i(\lambda_i) \le \1_{k_i}, |\lambda_i| = r_i  \text{ for } i = 1, 2 \} \\
    &= \bigsqcup_{r_1+r_2=r} Q_{r_1} ( \Lambda_1 ) \times Q_{r_2} (\Lambda_2).
\end{align*}

The second assertion follows from Remark~\ref{rmk:F.A.G. maps}.
\end{proof}

Recall from \cite[V.9]{MacLane} that if $K$ and $L$ are chain complexes with boundary maps
$\partial^K_r : K_r \to K_{r-1}$ and $\partial^L_r : L_r \to L_{r-1}$, then the tensor complex $K
\otimes L$ is given by
\[
(K \otimes L)_r = \bigoplus_{r_1 + r_2 = r} K_{r_1} \otimes L_{r_2},
\]
with boundary maps
\begin{equation}\label{eq:tensor boundary}
\partial^{K \otimes L}_{r_1 + r_2}(k \otimes l) := \partial^K_{r_1}(k) \otimes l + (-1)^{r_1} k \otimes \partial^L_{r_2}(l)
\quad\text{ for all } k \in K_{r_1}\text{ and }l \in L_{r_2}.
\end{equation}
The following is an analog of \cite[Theorem 2.7]{Grandis2005}.

\begin{prop} \label{prop:tensorcat}
Let $\Lambda_i$ be a $k_i$-graph for $i = 1,2$. The isomorphisms $\Psi_r$ of Lemma~\ref{lem:cube iso} induce an
isomorphism of complexes $\Psi : C_*(\Lambda_1 \times \Lambda_2) \to C_*(\Lambda_1) \otimes C_*(\Lambda_2)$.
\end{prop}
\begin{proof}
Fix $r_1, r_2$ such that $0 \le r_i \le k_i$ for $i=1, 2$ and set $r=r_1+r_2$. Let $\lambda_i \in
Q_{r_i}(\Lambda_i)$ ($i =1,2)$. Then for each $0 \le j \le k_1 + k_2$ and $\ell \in \{0,1\}$,
\[
F^\ell_j (\lambda_1 , \lambda_2)   =
    \begin{cases}
        (  F^\ell_j ( \lambda_1 ) , \lambda_2 ) &\text{ if $1 \le j \le r_1$} \\
        ( \lambda_1 ,  F^\ell_{j-r_1} (\lambda_2 ) ) &\text{ if $r_1 +1 \le j \le r_1 + r_2$}.
    \end{cases}
\]
Hence by~\eqref{eq:bondarydef} we may calculate:
\begin{align}
\partial_r(\lambda_1 , \lambda_2)
    &=  \sum_{\ell=0}^1 \sum^r_{j=1} (-1)^{\ell+j}  F_j^\ell ( \lambda_1 , \lambda_2 )  \nonumber \\
    &=  \sum_{\ell=0}^1 \Big(\sum_{j=1}^{r_1} (-1)^{\ell+j} ( F^\ell_j ( \lambda_1 ) , \lambda_2 )
        + \sum_{j= r_1+1}^{r_1 + r_2} (-1)^{\ell+j} ( \lambda_1 , F^\ell_{j-r_1} (\lambda_2 )) \Big) \nonumber \\
    &= \sum_{\ell=0}^1 \sum_{j=1}^{r_1} (-1)^{\ell+j}(  F^\ell_j ( \lambda_1 ) , \lambda_2 )
        +  \sum_{\ell=0}^1\sum_{h = 1}^{r_2} (-1)^{\ell + h + r_1} ( \lambda_1 ,  F^\ell_{h} (\lambda_2 ) )
        \nonumber \\
    &= ( \partial_{r_1}(\lambda_1), \lambda_2 ) + (-1)^{r_1} ( \lambda_1 , \partial_{r_2}(\lambda_2) )
 \label{eq:partialr} .
\end{align}
It remains to show that for all $r$,
\[
\partial_r(\Psi_r ( \lambda_1 , \lambda_2 ) ) =
\Psi_{r-1}( \partial_r ( \lambda_1 , \lambda_2 ) ).
\]
By definition of the boundary map $\partial_r$ on $C_{r_1} ( \Lambda_1 ) \otimes C_{r_2} (
\Lambda_2 )$ (see~\eqref{eq:tensor boundary}), we have
\begin{align*}
\partial_r(\Psi_r(\lambda_1 , \lambda_2 ))
    &= \partial_r(\lambda_1 \otimes \lambda_2) \\
    &= \partial_{r_1}(\lambda_1) \otimes \lambda_2 + (-1)^{r_1} \lambda_1 \otimes \partial_{r_2}(\lambda_2) \\
    &= \Psi_{r-1}( \partial_{r_1}(\lambda_1), \lambda_2 ) + (-1)^{r_1} ( \lambda_1 , \partial_{r_2}(\lambda_2) ),
\end{align*}
and this is equal to $\Psi_{r-1}( \partial_r ( \lambda_1 , \lambda_2 ) )$ by~\eqref{eq:partialr}.
\end{proof}

We may now state a K\"{u}nneth formula for our homology. The map $\alpha$ was considered in
\cite[Theorem 2.7]{Grandis2005}.

\begin{thm} \label{thm:kunneth}
Let $\Lambda_i$ be a $k_i$-graph for $i = 1,2$.  Then there is a split exact
sequence
\[
0 \to \bigoplus_{r_1 +r_2 = r} H_{r_1} ( \Lambda_1 ) \otimes H_{r_2} ( \Lambda_2 )
\xrightarrow{\alpha}  H_r ( \Lambda_1 \times \Lambda_2 ) \xrightarrow{\beta}
\bigoplus_{r_1 +r_2 = r-1} \Tor\big(H_{r_1} ( \Lambda_1 ), H_{r_2} ( \Lambda_2 )\big) \to 0.
\]
The homomorphisms $\alpha$ and $\beta$ are natural with respect to maps induced by $k$-graph
morphisms, but the splitting is not natural.
\end{thm}

\begin{proof}
The result follows from Proposition \ref{prop:tensorcat} and \cite[Theorem V.10.4]{MacLane} using the fact
that $C_r ( \Lambda )$ is torsion free for each $r$.
\end{proof}

\begin{cor} \label{cor:alphaisiso}
Let $\Lambda_i$ be a $k_i$-graph for $i = 1,2$. Suppose that for some $i$ the groups $H_r ( \Lambda_i )$ are all torsion-free.
Then the map $\alpha$ in Theorem \ref{thm:kunneth} is an isomorphism, so
\[
H_r ( \Lambda_1 \times \Lambda_2 )
    \cong \bigoplus_{r_1 +r_2 = r} H_{r_1} ( \Lambda_1 ) \otimes H_{r_2} ( \Lambda_2 ).
\]
\end{cor}

\begin{example}\label{ex:torus}
For $k \ge 2$, we have $T_k \cong T_1 \times \cdots \times T_1$ by Examples~\ref{ex:prodex} (1).
We claim that for  $0 \le r \le k$ we have
\[
H_r ( T_k ) \cong \ZZ^{\binom{k}{r}} .
\]
For $k = 0, 1$ this follows by Examples \ref{ex:donebydef}. The general case follows by induction on $k$ using
Corollary~\ref{cor:alphaisiso}.
\end{example}

\begin{dfn}
We say that a $k$-graph $\Lambda$ is \emph{acyclic} if $H_0(\Lambda) \cong \ZZ$ and $H_r(\Lambda)
= 0$ for all $r \ge 1$.
\end{dfn}

\begin{rmk} \label{rmk:acyclic}
Let $\Lambda_i $ be an acyclic $k_i$-graph for $i=1,2$. Then by Corollary \ref{cor:alphaisiso}
it follows that $\Lambda_1 \times \Lambda_2$ is an acyclic $k_1+k_2$-graph.
\end{rmk}

\begin{examples} \label{ex:acyclicex}
\begin{enumerate}
\item Note that by Examples~\ref{ex:prodex} (2) we have $\Delta_k \cong \Delta_1 \times \cdots \times \Delta_1$ for $k \ge 2$. By Example \ref{ex:delta1}
$\Delta_1$ is acyclic, and so by Remark \ref{rmk:acyclic} it follows that $\Delta_k$ is acyclic for all $k$. Indeed for $k \ge 1$ the $k$-graph $\Delta_k$ has the same
homology as $\RR^k$.

\item Let $\Lambda$ be a connected $1$-graph which is a tree. By Proposition
    \ref{prop:H0ofconn} we have $H_0 ( \Lambda ) \cong \ZZ$. Since $\Lambda$ contains no
    closed undirected paths, $C_1 ( \Lambda )$ has no closed trails.
   Thus by Proposition
    \ref{prop:trailsforker} $\ker  ( \partial_1 ) = 0$ and hence $H_1 ( \Lambda ) = 0$.
Since $H_r ( \Lambda )= 0$ for $r >1$, it follows that $\Lambda$ is acyclic.
\end{enumerate}
\end{examples}

The proof of the next result follows the argument used in~\cite[II.4.1]{Brown}. This result may also be
deduced from \cite[Theorem 3.3]{Grandis2005} using the identification of our homology with that of the
corresponding cubical set established in Theorem~\ref{thm:complexiso}.

\begin{thm} \label{thm:freeaction}
Suppose that $\Lambda$ is an acyclic $k$-graph. If $G$ is a
discrete group acting freely on $\Lambda$, then $H_*(\Lambda/G)
\cong H_*(G, \ZZ)$.
\end{thm}

\begin{proof}
If $M$ is a $G$-module, then we write $DM$ for
the submodule of $M$ generated by the elements $\{ gm - m : m
\in M, g \in G\}$. We write $M_G$ for $M/DM$.  Note that  $M \mapsto M_G$ is
a functor from the category of $G$-modules to the category of abelian groups
(so it maps a complex of $G$-modules to a complex of abelian groups).
If $G$ acts on a set $X$ then $\ZZ X$ may be regarded as a $G$-module and
$\ZZ X_G \cong \ZZ(X/G)$ (see \cite[\S\,II.2]{Brown}).

Since $G$ acts freely on $\Lambda$, it acts freely on each $Q_r(\Lambda)$. Thus $C_r( \Lambda ) = \ZZ Q_r( \Lambda )$ is
a free $G$-module. We have
\[
\ZZ Q_r(\Lambda)_G
\cong \ZZ Q_r(\Lambda/G).
\]
Moreover, this isomorphism is compatible with the boundary maps. So if $C_*(\Lambda)_G$ denotes the complex obtained from
$C_*(\Lambda)$ by applying the functor $M \mapsto M_G$, then $C_*(\Lambda)_G \cong C_*(\Lambda/G)$. Since
$\Lambda$ is acyclic, the sequence
\[
\dots \stackrel{\partial_3}{\to} C_2(\Lambda) \stackrel{\partial_2}{\to} C_1(\Lambda)
\stackrel{\partial_1}{\to} C_0(\Lambda) \stackrel{\varepsilon}{\to} \ZZ \to 0
\]
is a resolution of $\ZZ$ by free $G$-modules. Since the complex $C_*(\Lambda)_G$ is isomorphic to the complex
$C_*(\Lambda/G)$, we have
\[
H_*(C_*(\Lambda)_G) \cong H_*(C_*(\Lambda/G)) = H_*(\Lambda/G).
\]
Therefore, $H_*(G,\ZZ) \cong H_*(\Lambda/G)$.
\end{proof}

Recall that  the fundamental group $\pi_1 ( \Lambda )$ of a connected $1$-graph $\Lambda$ is free (see for example  \cite[\S
2.1.8]{Stillwell} or \cite[\S 4]{KumjianPask:ETDS99}) and the universal cover $T$ is a tree.  Thus $\Lambda$ may be realised
as the quotient of $T$  by  the action of $\pi_1 ( \Lambda )$; moreover, if $ \Lambda$ has finitely many vertices and edges, then
$\pi_1 ( \Lambda ) \cong \FF_p$, where $\FF_p$ is the free group on $p$ generators and $p = \vert  \Lambda^1 \vert -\vert
\Lambda^0 \vert +1$ (see \cite[\S I.3.3, Theorem 4]{Serre}). Since $T$ is acyclic, we obtain the following result.

\begin{cor} \label{cor:1graph}
Let $\Lambda$ be a connected $1$-graph. Then $H_1 ( \Lambda ) \cong H_1 ( \pi_1(\Lambda) , \ZZ )$. In particular if
$\Lambda$ has finitely many vertices and edges, then $ \pi_1(\Lambda) \cong \FF_p$ where $p = | \Lambda^1 | - |
\Lambda^0 | +1$ and so
\[
H_1(\Lambda )  \cong H_1 (\FF_p, \ZZ ) \cong \ZZ^p.
\]
\end{cor}

\begin{examples} \label{ex:fga}
\begin{enumerate}
\item\label{it:Cayley} Recall from Examples~\ref{ex:kgraphs} (\ref{it:Bsubn}) that $B_n$ is the path category of a directed
    graph with a single vertex and $n$ edges, regarded as a $1$-graph. The universal cover of $B_n$, which we denote $A_n$, is
    the skew-product $B_n \times_{c} \FF_n$, and can be identified with the Cayley graph of $\FF_n$, the free group on $n$
    generators. By \cite[Remark 5.6]{KP2000} $\FF_n$ acts freely on $A_n$ with $ A_n/\FF_n\cong B_n$. By
    Corollary~\ref{cor:1graph} we have $H_1 ( B_n ) \cong \ZZ^n$. Hence $B_n$ has the same homology as the wedge of $n$
    circles.

\item Let $H$ be a subgroup of $\ZZ^k$. Then as in \cite[\S 6.4]{kps} $H$ acts freely on
    $\Delta_k$. Since $\Delta_k$ is acyclic, by Theorem \ref{thm:freeaction} we have $H_* (
    \Delta_k / H ) \cong H_* ( H , \ZZ )$.  If $H \cong \ZZ^q$, then for $0 \le r \le k$ we
    have (cf.\ Example~\ref{ex:torus})
    \[
        H_r ( \Delta_k / H ) \cong \ZZ^{\left( \begin{smallmatrix} q \\ r \end{smallmatrix} \right)} .
    \]
    Hence $\Delta_k / H$ has the same homology as the $q$-torus. If $H$ has finite index then $q = k$ and the quotient graph
    $\Delta_k / H$ may be viewed as yet another $k$-graph analog of the $k$-torus
     (note $\Delta_k / H = T_k$ when $H = \ZZ^k$).

\item The following example indicates that Theorem~\ref{thm:freeaction} is in practise less
    useful than it might appear because it is difficult to recognise acyclic $k$-graphs (short
    of explicitly computing their homology). In particular one might expect that a pullback of
    an acyclic $k$-graph by a full-rank endomorphism of $\NN^k$ is itself acyclic, but this is
    not so.

    Let $\Lambda$ be the $2$-graph with $\Lambda^0 = \{v\}$,  $\Lambda^{e_1} = \{a_1, a_2\}$, $\Lambda^{e_2} =
    \{b_1, b_2\}$, and factorisation property determined by $a_ib_j = b_ia_j$ for $i, j = 1, 2$. Recall that we denote the
    generators of $\FF_2$ by $h_1$ and $h_2$. There is a functor $\sigma : \Lambda \to \FF_2 \times \ZZ$ determined by
    $\sigma(a_i) = (h_i, 0)$ and $\sigma(b_i) = (h_i, 1)$. Let $\Gamma := \Lambda \times_\sigma (\FF_2 \times \ZZ)$,
    and observe that by \cite[Remark 5.6]{KP2000} $\FF_2 \times \ZZ$ acts freely on $\Gamma$ with quotient $\Lambda$.

    Let $A_2 = B_2 \times_c \FF_2$ as in~(\ref{it:Cayley}) above. Define $g : \NN^2 \to \NN^2$ by $g(m,n) := (m+n, n)$.
    Tedious calculations show that $\Gamma$ is isomorphic to the pullback $g^*(A_2 \times \Delta_1)$\footnote{This is
    not meant to be obvious. After unraveling the definitions of $\Gamma$ and of $g^*(A_2 \times \Delta_1)$, one can check
    that the formulas $(a_i, (h,n)) \mapsto \big(((f_i, h),(n,n)), (1,0)\big)$ and $(b_i, (h,n)) \mapsto \big(((f_i, h),(n,n+1)),
    (0,1)\big)$ for $i=1,2$ determine the desired isomorphism.}.

    We claim that $\Gamma$ is not acyclic.  Suppose that it is. Then Theorem~\ref{thm:freeaction}
    implies that $H_*(\Lambda) \cong H_*(\FF_2 \times \ZZ)$. By the K\"unneth theorem for group
    homology, since both $H_r(\FF_2)$ and $H_r(\ZZ)$ are trivial for $r \ge 2$,
    \[
        H_2(\FF_2 \times \ZZ) = H_1(\FF_2) \otimes H_1(\ZZ) \cong \ZZ^2.
    \]
    A straightforward computation shows that $a_1b_1$, $a_2b_2$ and $a_1b_2 + a_2b_1$ all belong to $\ker(\partial_2) =
    H_2(\Lambda)$, so the latter has rank at least three, giving a contradiction.

    So $\Gamma$ is not acyclic, despite being a pull-back of the acyclic graph $A_2 \times \Delta_1$
    (see Remark~\ref{rmk:acyclic}) by the full-rank endomorphism $g$.
\end{enumerate}
\end{examples}

We now turn our attention to exact sequences of homology groups associated to automorphisms of $k$-graphs. Recall from
\cite{FPS2009} that if $\Lambda$ is a $k$-graph and $\alpha$ is an automorphism of $\Lambda$, then there is a
$(k+1)$-graph $\Lambda \times_\alpha \ZZ$\label{pg:cpgraph} with morphisms $\Lambda \times \NN$, range and source
maps given by $r(\lambda,n) = (r(\lambda),0)$, $s(\lambda,n) = (\alpha^{-n}(s(\lambda)), 0)$, degree map given by
$d(\lambda,n) = (d(\lambda),n)$ and composition given by $(\lambda,m)(\mu,n) := (\lambda\alpha^m(\mu), m+n)$. In
particular $( \Lambda \times_\alpha \ZZ )^0 = \Lambda^0 \times \{ 0 \}$.

We may describe the cubes of $\Lambda \times_\alpha \ZZ$ in terms of those of $\Lambda$ as follows:
$Q_0 ( \Lambda \times_\alpha \ZZ ) = Q_0 ( \Lambda ) \times \{ 0 \}$ and for $0 \le r \le k$ an
element of $Q_{r+1} ( \Lambda \times_\alpha \ZZ )$ is of the form $(\lambda , 0 )$ where $\lambda
\in Q_{r+1} ( \Lambda )$ or $( \lambda , 1 )$ where $\lambda \in Q_{r} ( \Lambda )$, so
\begin{equation} \label{eq:QrLambdatimesAlpha}
Q_{r+1}(\Lambda \times_\alpha \ZZ) = (Q_{r+1}(\Lambda) \times \{0\}) \sqcup (Q_{r}(\Lambda) \times \{1\}).
\end{equation}
Given an element $a = \sum a_\lambda \lambda \in C_r(\Lambda)$, we shall somewhat inaccurately
write $(a,0) := \sum a_\lambda (\lambda,0)$ and $(a,1) := \sum a_\lambda (\lambda,1)$ for the
corresponding elements of $C_r(\Lambda \times_\alpha \ZZ)$ and $C_{r+1}(\Lambda \times_\alpha
\ZZ)$. With this notation, the boundary map on $C_{r+1}(\Lambda \times_\alpha \ZZ)$ is given by
\begin{equation}\label{eq:cp boundary}
\begin{array}{rcl}
   \partial_{r+1}(\lambda,0) &=& (\partial_{r+1}(\lambda), 0)
        \quad\text{ and }\\
    \partial_{r+1}(\mu,1) &=& (-1)^{r}\big((\alpha^{-1}(\mu), 0) - (\mu,0)\big) + (\partial_r(\mu),1).
\end{array}
\end{equation}

We will deduce our long exact sequence for the homology of $\Lambda \times_\alpha \ZZ$ from the long exact sequence
associated to a mapping-cone complex arising from the chain map $\alpha^{-1} - 1$ (see \cite[Proposition~II.4.3]{MacLane}).
So we recall the definition of the mapping cone complex. Given a chain map $f : A_* \to B_*$, define a complex $M_* = M(f)_*$
by $M_r := A_{r-1} \oplus B_r$ (with the convention that $A_{-1} = \{0\}$) with boundary map
\begin{equation}\label{eq:mapping cone boundary}
    \partial_r(a, b) := (-\partial_{r-1}(a), \partial_r(b) + f(a)).
\end{equation}

If $\alpha$ is an automorphism of a $k$-graph $\Lambda$, then $\alpha^{-1}$ maps cubes to cubes and
intertwines boundary maps, and so induces a chain map $\alpha^{-1} : C_*(\Lambda) \to
C_*(\Lambda)$. Hence $\alpha^{-1} - 1$ is also a chain map from $C_*(\Lambda)$ to itself.

\begin{lem}\label{lem:cp mapping cone}
Let $\Lambda$ be a $k$-graph and let $\alpha$ be an automorphism of $\Lambda$. Then there is an
isomorphism of chain complexes $\psi : C_*(\Lambda \times_\alpha \ZZ) \to M(\alpha^{-1} - 1)_*$
such that
\[
\psi(\lambda,0) = (0, \lambda)\quad\text{ and }\quad \psi(\mu,1) = ((-1)^{r} \mu, 0)
\]
for all $(\lambda,0), (\mu,1) \in Q_{r+1}(\Lambda \times_\alpha \ZZ)$.  Hence, $\psi_*: H_*(\Lambda
\times_\alpha \ZZ) \to H_*(M(\alpha^{-1} - 1)_*)$ is an isomorphism.
\end{lem}
\begin{proof}
Write $M_* := M(\alpha^{-1} - 1)_*$ and $C_* := C_*(\Lambda \times_\alpha \ZZ)$. It is clear that
$\psi$ determines isomorphisms of groups $C_r \cong M_r$. So to see that $\psi$ is an isomorphism
of complexes, it suffices to show that it intertwines the boundary maps on generators. We consider
cubes of the form $(\lambda,0)$ and those of the form $(\mu,1)$ separately. Fix $\lambda \in
Q_{r+1}(\Lambda)$. We have $\partial_{r+1}(\psi(\lambda,0)) = \partial_{r+1}(0,\lambda) = (0,
\partial_{r+1}(\lambda))$ by~\eqref{eq:mapping cone boundary}, and $\psi(\partial_{r+1}(\lambda,0)) =
\psi(\partial_{r+1}(\lambda), 0) = (0, \partial_{r+1}(\lambda))$ by~\eqref{eq:cp boundary}. So
$\partial_{r+1}(\psi(\lambda,0)) = \psi(\partial_{r+1}(\lambda,0))$ as required.

Now fix $\mu \in Q_r(\Lambda)$. Then we have
\begin{flalign*}
&&\psi(\partial_{r+1}(\mu,1))
        &= \psi\big((-1)^{r}\big((\alpha^{-1}(\mu), 0) - (\mu,0)\big) + (\partial_r(\mu),1)\big) &\\
        &&&= (-1)^r \big(\psi(\alpha^{-1}(\mu), 0) - \psi(\mu,0)\big) + \psi(\partial_r(\mu),1) &\\
        &&&= (-1)^r (0, (\alpha^{-1}-1)(\mu)) + (-1)^{r-1}(\partial_r(\mu),0) &\\
        &&&= (-1)^r (-\partial_r(\mu), (\alpha^{-1}-1)(\mu)).&
\intertext{On the other hand,}
&& \partial_{r+1}\psi((\mu,1))
        &= (-1)^r\partial_{r+1}(\mu,0) &\\
        &&&= (-1)^r (-\partial_{r}(\mu), \partial_{r+1}(0) + (\alpha^{-1} - 1)(\mu)) &\\
        &&&= (-1)^r (-\partial_r(\mu), (\alpha^{-1}-1)(\mu)). &\qedhere
\end{flalign*}
\end{proof}

Now recall from \cite[Proposition~4.3]{MacLane}, that a chain map $f : A_* \to B_*$ determines a
long exact sequence
\begin{equation}\label{eq:mapping cone sequence}
 \cdots  \to H_r(B_*) \xrightarrow{\iota_*} H_{r}(M(f)_*) \xrightarrow{\pi_*}
    H_{r-1}(A_*) \xrightarrow{f_*} H_{r-1}(B_*)  \to \cdots
\end{equation}
where $\iota_* : H_r(B_*) \to H_r(M(f)_*)$ is induced by the inclusion map $\iota : B_r \to M(f)_r$, and $\pi_* : H_r(M(f)_*)
\to H_{r-1}(A_*)$ is induced by the projection $\pi : M(f)_r \to A_{r-1}$.

The following result gives an exact sequence which may be regarded as an analog of the Pimsner-Voiculescu
sequence for crossed products of $C^*$-algebras (cf. \cite[Theorem~2.4]{PV1980},
\cite[Theorem~10.2.1]{Blackadar}).

\begin{thm}\label{thm:PV}
Let $\Lambda$ be a $k$-graph, and let $\alpha$ be an automorphism of $\Lambda$.  Then there is an
exact sequence
\[\begin{split}
    0 \longrightarrow H_{k+1} ( \Lambda \times_\alpha \ZZ  ) &\stackrel{\pi_*}{\longrightarrow}
        H_k ( \Lambda  ) \xrightarrow{1-\alpha_*} H_k ( \Lambda ) \stackrel{\iota_*}{\longrightarrow}
        H_k(\Lambda \times_\alpha \ZZ) \longrightarrow \cdots \\
    \cdots &\longrightarrow H_1 ( \Lambda \times_\alpha \ZZ ) \stackrel{\pi_*}{\longrightarrow} H_0(\Lambda)
        \xrightarrow{1-\alpha_*} H_0(\Lambda) \stackrel{\iota_*}{\longrightarrow}
        H_0(\Lambda \times_\alpha \ZZ) \longrightarrow 0.
\end{split}\]
\end{thm}
\begin{proof}
The long exact sequence~\eqref{eq:mapping cone sequence} applied with $f = \alpha^{-1} - 1$
together with Lemma~\ref{lem:cp mapping cone} (and identifying
$H_*(\Lambda \times_\alpha \ZZ) \cong H_*(M(\alpha^{-1} - 1)_*)$) gives a long exact sequence
\[\begin{split}
    0 \longrightarrow H_{k+1} ( \Lambda \times_\alpha \ZZ  ) &\stackrel{\pi_*}{\longrightarrow}
        H_k ( \Lambda  ) \xrightarrow{\alpha^{-1}_* - 1} H_k ( \Lambda ) \stackrel{\iota_*}{\longrightarrow}
        H_k(\Lambda \times_\alpha \ZZ) \longrightarrow \cdots \\
    \cdots &\longrightarrow H_1 ( \Lambda \times_\alpha \ZZ ) \stackrel{\pi_*}{\longrightarrow} H_0(\Lambda)
        \xrightarrow{\alpha^{-1}_* - 1} H_0(\Lambda) \stackrel{\iota_*}{\longrightarrow}
        H_0(\Lambda \times_\alpha \ZZ) \longrightarrow 0.
\end{split}\]
Since $\alpha_*$ is an automorphism of $H_r(\Lambda)$ which commutes with $\alpha^{-1}_* - 1$, both
$\ker(\alpha^{-1}_* - 1)$ and $\image(\alpha^{-1}_* - 1)$ are $\alpha_*$-invariant. Therefore
\[
\ker(\alpha^{-1}_* - 1) = \ker(\alpha_*(\alpha^{-1}_* - 1)) = \ker(1 - \alpha_*)
\]
and similarly, $\image(\alpha^{-1}_* - 1) = \image(1 -\alpha_*)$.
\end{proof}

\begin{rmk}
Theorem~\ref{thm:PV} may also be proved using the topological realizations, introduced in
\cite{kkqs} (see also Section~\ref{sec:geometric}), of $\Lambda$ and $\Lambda \times_\alpha \ZZ$.
To see how, recall from \cite[Lemma~2.23]{kkqs} that $\alpha$ induces a homeomorphism
$\tilde{\alpha}$ of the topological realisation $X_\Lambda$ of $\Lambda$, and that $X_{\Lambda
\times_\alpha \ZZ}$ is homeomorphic to the mapping torus $M(\tilde{\alpha})$.  Combining this with
Theorem~\ref{thm:homologies agree} and the long exact sequence of
\cite[Example~2.48]{Hatcher:AlgTop02} yields the result.
\end{rmk}

\section{Examples}\label{sec:examples}

In this section we present some examples. We describe them using skeletons, so we first indicate
what this means. Our examples are all $2$-graphs (since there are already a number of interesting
examples in this case), so we restrict ourselves to a discussion of skeletons for $2$-graphs.

A $2$-coloured graph is a directed graph $E$ together with a map $c : E^1 \to \{1,2\}$. A complete
collection of squares in $E$ is a collection of relations of the form $ef \sim f'e'$ where $ef,
f'e' \in E^2$ with $c(e) = c(e') = 1$ and $c(f) = c(f') = 2$ such that each bi-coloured path of
length two appears in exactly one such relation%
\footnote{Strictly speaking, in \cite{HazelwoodRaeburnEtAl:xx11}, a complete collection of squares is defined to be a collection
$\Cc$ of coloured-graph morphisms from model coloured graphs $E_{k,e_i + e_j}$ into $\Lambda$, and the relation $\sim$ is
defined by $ef \sim f'e'$ if and only if the two paths traverse a common element of $\Cc$. But we can recover the
collection of coloured-graph morphisms as in \cite{HazelwoodRaeburnEtAl:xx11} from the relation $\sim$, so the two formalisms are equivalent.}%
. It follows from \cite[Section~6]{KP2000} (see also \cite[Theorems
4.4~and~4.5]{HazelwoodRaeburnEtAl:xx11}) that each pair consisting of a $2$-coloured graph and a
complete collection of pairs uniquely determines a $2$-graph, and also that each $2$-graph arises from such
a pair $(E_\Lambda, \Cc_\Lambda)$. It is standard to refer to the equalities $ef = f'e'$ in $\Lambda$
determined by the squares $ef \sim f'e'$ in $\Cc$ as the \emph{factorisation rules.} We refer to $E$ as
the \emph{skeleton} of $\Lambda$.

In our diagrams, edges of colour $1$ are blue and solid, and edges of colour $2$ are red and
dashed.

Our first example is a $2$-graph whose first homology group contains torsion. Combined with Example~\ref{ex:depends on
factorisations}, it also demonstrates that the homology of a $k$-graph depends on the factorisation rules and not just on the
skeleton.

\begin{example}\label{ex:torsion}
Fix $n > 1$ and consider the $1$-graph $\Lambda$ with skeleton
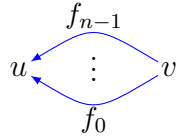
\[
\begin{tikzpicture}[scale=2]
    \node[circle,inner sep=0.5pt] (v) at (0,0) {$u$};
    \node[circle,inner sep=0.5pt] (w) at (1,0) {$v$};
    \draw[-latex, blue] (w) .. controls +(-0.5,-0.3) .. (v) node[pos=0.5,anchor=north,inner sep=1pt, black]
    {\small$f_0$};
    \draw[-latex, blue] (w) .. controls +(-0.5,+0.3) .. (v) node[pos=0.5,anchor=south,inner sep=1pt, black]
    {\small$f_{n-1}$};
    \node at (0.5,0.05) {$\vdots$};
\end{tikzpicture}
\]

Define $\alpha \in \Aut(\Lambda)$ by $\alpha(f_i) = f_{i+1}$, where addition is modulo $n$ (so $\alpha$ fixes vertices).
Then $\Lambda \times_\alpha \ZZ$ (see page~\pageref{pg:cpgraph}) is the 2-graph with skeleton
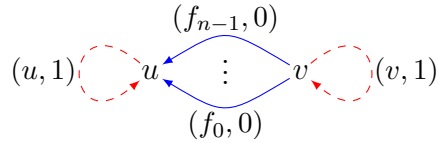
\[
\begin{tikzpicture}[scale=2]
    \node[circle,inner sep=0.5pt] (v) at (0,0) {$u$};
    \node[circle,inner sep=0.5pt] (w) at (1,0) {$v$};
    \draw[-latex, blue] (w) .. controls +(-0.5,-0.3) .. (v) node[pos=0.5,anchor=north,inner sep=1pt, black] {\small$(f_0,0)$};
    \draw[-latex, blue] (w) .. controls +(-0.5,+0.3) .. (v) node[pos=0.5,anchor=south,inner sep=1pt, black] {\small$(f_{n-1},0)$};
    \node at (0.5,0.05) {$\vdots$};
    \draw[-latex,red,dashed] (v) .. controls +(-0.6,0.5) and +(-0.6,-0.5) .. (v) node[pos=0.5,anchor=east,inner sep=1pt, black] {\small$(u,1)$};
    \draw[-latex,red, dashed] (w) .. controls +(0.6,0.5) and +(0.6,-0.5) .. (w) node[pos=0.5,anchor=west,inner sep=1pt, black] {\small$(v,1)$};
\end{tikzpicture}
\]
and factorisation rules $(f_i , 0 ) (v,1) = (u,1) ( f_{i+1} , 0)$ for $i =0, \ldots, n-1$, where
addition is modulo $n$.

We claim that
\[
H_0(\Lambda \times_\alpha \ZZ) \cong \ZZ,\quad
    H_1(\Lambda \times_\alpha \ZZ) \cong \ZZ \oplus \ZZ/n\ZZ,\quad\text{ and }\quad
    H_2(\Lambda \times_\alpha \ZZ) = \{0\}.
\]
By Proposition~\ref{prop:H0ofconn} we have $H_0(\Lambda \times_\alpha \ZZ) \cong \ZZ$ and $H_0(\Lambda) \cong
\ZZ$. Since $\alpha$ fixes vertices it follows that $\alpha_* : H_0 (\Lambda) \to H_0 (\Lambda)$ is the identity map. Hence
$\ker(1 - \alpha_*) = H_0(\Lambda) \cong \ZZ$.

We next calculate $H_1(\Lambda)$. Since $C_2(\Lambda) = \{0\}$, we have $H_1(\Lambda) = \ker(\partial_1)$. Since
$\partial_1( f_i) = u -v$ for all $0 \le i \le n-1$, and since $C_1(\Lambda) = \ZZ \{f_0, \dots , f_{n-1}\}$, we have
\begin{equation}\label{eq:old basis for H_1}
\{f_i - f_{i+1} : 0 \le i \le n-2\}\text{ is a basis for the $\ZZ$-module $H_1(\Lambda)$.}
\end{equation}
Let $b_i := f_i - f_{i+1}$ for $0 \le i \le n-2$
then $\alpha_*(b_i) = b_{i+1}$ for $0 \le i < n-2$, and
\[
\alpha_* ( b_{n-2} ) = f_{n-1} - f_0 = -\sum_{i=0}^{n-2} b_i.
\]
Hence, regarded as an endomorphism of $\ZZ^{n-1}$, the map $1 - \alpha_*$ is implemented
by the $(n-1) \times (n-1)$ matrix
\[
\left(\begin{array}{rrrcrr}
1 & 0 & 0 & \cdots & 0 & 1\\
-1  & 1 & 0 & \cdots & 0 & 1 \\
0 & -1 & 1 &\cdots & 0 & 1\\
\vdots\; &\vdots\;&\vdots\;&\ddots &\vdots\;&\vdots\; \\
0&0&0&\cdots&1&1 \\
0&0&0&\cdots&-1&2
\end{array}\right) .
\]
Thus $\image(1 - \alpha_*)$ is spanned by the elements $b_i - b_{i+1}$ for $0 \le i \le n-3$ together with the element
$b_{n-2} + \sum^{n-2}_{i=0} b_i$. Using this one checks that
\begin{equation}\label{eq:basis for image}
    \{b_0 - b_{n-2}, b_1 - b_{n-2}, \dots, b_{n-3} - b_{n-2}, nb_{n-2}\}\text{ is a basis for $\image(1 - \alpha_*)$.}
\end{equation}
From \eqref{eq:old basis for H_1} one sees that
\begin{equation}\label{eq:basis for H_1}
     \{b_0 - b_{n-2}, b_1 - b_{n-2}, \dots, b_{n-3} - b_{n-2}, b_{n-2}\}\text{ is a basis for $H_1(\Lambda)$.}
\end{equation}

In particular, $\rank(\image(1 - \alpha_*)) = \rank(H_1(\Lambda))$, forcing $\ker(1-\alpha_*) =
\{0\}$. Moreover, combining~\eqref{eq:basis for H_1} with~\eqref{eq:basis for image} shows that
$\coker(1 - \alpha_*) \cong \ZZ/n\ZZ$. Thus Theorem~\ref{thm:PV} implies that $H_2(\Lambda
\times_\alpha \ZZ) = \{0\} = \ker(1-\alpha_*) = \{0\}$, and that $H_1(\Lambda \times_\alpha \ZZ)$
is an extension of $\ZZ$ by $\ZZ/n\ZZ$ and hence is equal to $\ZZ \oplus
(\ZZ/n\ZZ)$.  In particular, for $n=2$, the graph $\Lambda \times_\alpha \ZZ$ has the
same homology as the Klein bottle.
\end{example}

\begin{example}\label{ex:depends on factorisations}
Let $T_1$ be the $1$-graph with a single vertex and a single edge as in Example~\ref{ex:kgraphs}(\ref{it:Tsubk}), and let
$\Lambda$ and $\Gamma = \Lambda \times_\alpha \ZZ$ be as in Example~\ref{ex:torsion} with $n=2$. Then $\Lambda
\times T_1$ has the same skeleton as $\Gamma$. To compute the homology of $\Lambda \times T_1$, we can use the
K\"unneth theorem (Theorem~\ref{thm:kunneth}): each of $T_1$ and $\Lambda$ consists of a single simple closed undirected
path, so it is routine to verify that $H_i(T_1) = H_i(\Lambda) = \ZZ$ for each of $i = 0, 1$. Hence $H_i(\Lambda \times
T_1) = \ZZ^{\binom{2}{i}}$ for all $i$. So the homology of $\Lambda \times T_1$ is the same as that of the $2$-torus (see
Example~\ref{ex:torus}), and in particular is not equal to that of $\Gamma$, even though they have the same skeleton.
\end{example}

We next describe a suite of examples of $2$-graphs whose homology mirrors that of the sphere, the torus, the Klein bottle and the
projective plane. We have presented examples matching the Klein bottle and the torus previously (see Examples
\ref{ex:torsion}~and~\ref{ex:torus}), but we provide presentations here which suggest standard planar diagrams for these four
spaces.

\begin{rmk}\label{rmk:planar diagrams}
For a number of the following examples, we give a non-standard presentation of the skeleton and factorisation rules. Specifically for
Examples \ref{ex:sphere}--\ref{ex:Klein}, we present a commuting diagram (in the category $\Lambda$) which includes all
$2$-cubes as commuting squares. These diagrams are not the same as the skeletons because they involve some repeated vertices and
edges. We present our examples this way to suggest planar diagrams for their topological realisations (see
Section~\ref{sec:geometric}); indeed, we will sometimes refer to these commuting diagrams, very imprecisely, as planar diagrams for
the associated $2$-graphs.

When using this presentation of a $2$-graph, one must check that the collection of squares specified in the diagram is complete:
since vertices may be repeated in a planar diagram, it is possible that there are some bi-coloured paths in the skeleton which do not
appear as the sides of a square in the diagram, and in this case, the diagram may not completely specify a $2$-graph, and is in any
case not a planar diagram for the $2$-graph in the sense just discussed.
\end{rmk}

\begin{example}\label{ex:sphere}
Let $\Lambda$ be the $2$-graph described by the following planar diagram (see
Remark~\ref{rmk:planar diagrams}).
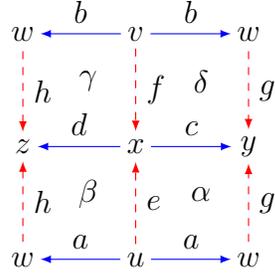
\[
\begin{tikzpicture}[scale=1.5]
    \node[inner sep=1pt, circle] (sw) at (-1,-1) {$w$};
    \node[inner sep=1pt, circle] (w) at (-1,0) {$z$};
    \node[inner sep=1pt, circle] (nw) at (-1,1) {$w$};
    \node[inner sep=1pt, circle] (s) at (0,-1) {$u$};
    \node[inner sep=1pt, circle] (m) at (0,0) {$x$};
    \node[inner sep=1pt, circle] (n) at (0,1) {$v$};
    \node[inner sep=1pt, circle] (se) at (1,-1) {$w$};
    \node[inner sep=1pt, circle] (e) at (1,0) {$y$};
    \node[inner sep=1pt, circle] (ne) at (1,1) {$w$};
    \draw[-latex, blue] (s)--(sw) node[pos=0.5, above] {\color{black}$a$};
    \draw[-latex, blue] (s)--(se) node[pos=0.5, above] {\color{black}$a$};
    \draw[-latex, blue] (m)--(w) node[pos=0.5, above] {\color{black}$d$};
    \draw[-latex, blue] (m)--(e) node[pos=0.5, above] {\color{black}$c$};
    \draw[-latex, blue] (n)--(nw) node[pos=0.5, above] {\color{black}$b$};
    \draw[-latex, blue] (n)--(ne) node[pos=0.5, above] {\color{black}$b$};
    \draw[-latex, red, dashed] (sw)--(w) node[pos=0.5, right] {\color{black}$h$};
    \draw[-latex, red, dashed] (nw)--(w) node[pos=0.5, right] {\color{black}$h$};
    \draw[-latex, red, dashed] (s)--(m) node[pos=0.5, right] {\color{black}$e$};
    \draw[-latex, red, dashed] (n)--(m) node[pos=0.5, right] {\color{black}$f$};
    \draw[-latex, red, dashed] (se)--(e) node[pos=0.5, right] {\color{black}$g$};
    \draw[-latex, red, dashed] (ne)--(e) node[pos=0.5, right] {\color{black}$g$};
    \node at (-0.42, -0.42) {$\beta$};
    \node at (-0.42, 0.58) {$\gamma$};
    \node at (0.58, -0.42) {$\alpha$};
    \node at (0.58, 0.58) {$\delta$};
\end{tikzpicture}
\]
The skeleton of $\Lambda$ is pictured in~\eqref{fig:sphere}. The Greek letters in the centres of the commuting squares in the
above diagram are the morphisms of degree $(1,1)$. So $\alpha = ce = ga$, $\beta = de = ha$, etc.

Since $\Lambda$  is connected, $H_0 ( \Lambda ) \cong \ZZ$ by Proposition~\ref{prop:H0ofconn}. We have
$\partial_{2}(\alpha - \beta + \gamma - \delta) = 0$ by a straightforward calculation and one can check that $\partial_2(n_1
\alpha + n_2\beta + n_3\gamma) = 0$ implies $n_1 = n_2 = n_3 = 0$, so $H_2(\Lambda) = \ker(\partial_2) \cong \ZZ$.
Moreover, $\partial_2(C_2(\Lambda))$ is spanned by $\partial_2(\alpha)$, $\partial_2(\beta)$ and $\partial_2(\gamma)$.

One checks that the set $\{\partial_2(\alpha), \partial_2(\beta), \partial_2(\gamma), d, e ,f ,g , h\}$ forms a basis for
$C_1(\Lambda)$. So $C_1(\Lambda) = \partial_2(C_2(\Lambda)) \oplus \ZZ\{d, e, f, g, h\}$. Since $H_0(\Lambda) =
\ZZ$ and $C_0(\Lambda)$ has rank 6, the image of $\partial_1$ has rank $5$. It follows that $H_1(\Lambda) = \{0\}$.
Hence $\Lambda$ has the same homology as the sphere $S^2$. If we draw its skeleton as follows, the resemblance between
$\Lambda$ and a combinatorial sphere is striking.
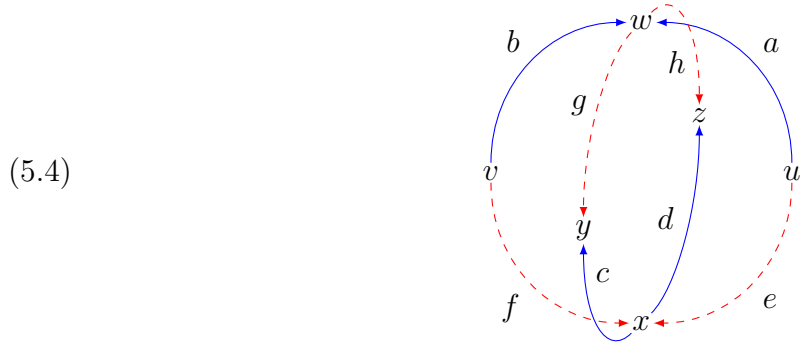
\begin{equation}\label{fig:sphere}
\parbox[c]{0.9\textwidth}{\hfill\begin{tikzpicture}[scale=2]
    \node[inner sep= 1pt] (100) at (1,0,0) {$u$};
    \node[inner sep= 1pt] (-100) at (-1,0,0) {$v$};
    \node[inner sep= 1pt] (010) at (0,1,0) {$w$};
    \node[inner sep= 1pt] (0-10) at (0,-1,0) {$x$};
    \node[inner sep= 1pt] (001) at (0,0,1) {$y$};
    \node[inner sep= 1pt] (00-1) at (0,0,-1) {$z$};
    \draw[-latex, blue] (100) .. controls +(0,0.6,0) and +(0.6,0,0) .. (010) node[pos=0.5, anchor=south west] {\color{black}$a$};
    \draw[-latex, red, dashed] (100) .. controls +(0,-0.6,0) and +(0.6,0,0) .. (0-10) node[pos=0.5, anchor=north west] {\color{black}$e$};
    \draw[-latex, blue] (-100) .. controls +(0,0.6,0) and +(-0.6,0,0) .. (010) node[pos=0.5, anchor=south east] {\color{black}$b$};
    \draw[-latex, red, dashed] (-100) .. controls +(0,-0.6,0) and +(-0.6,0,0) .. (0-10) node[pos=0.5, anchor=north east] {\color{black}$f$};
    \draw[-latex, red, dashed] (010) .. controls +(0,0,0.6) and +(0,0.6,0) .. (001) node[pos=0.5, anchor=east] {\color{black}$g$};
    \draw[-latex, red, dashed] (010) .. controls +(0,0,-0.6) and +(0,0.6,0) .. (00-1) node[pos=0.8, anchor=east] {\color{black}$h$};
    \draw[-latex, blue] (0-10) .. controls +(0,0,0.6) and +(0,-0.6,0) .. (001) node[pos=0.85, anchor=west] {\color{black}$c$};
    \draw[-latex, blue] (0-10) .. controls +(0,0,-0.6) and +(0,-0.6,0) .. (00-1) node[pos=0.5, anchor=south east] {\color{black}$d$};
\end{tikzpicture}\hfill\ }
\end{equation}
\end{example}

\begin{example}\label{ex:four-square torus}
Consider the $2$-graph $\Sigma$ with planar diagram (see Remark~\ref{rmk:planar diagrams}) on the left
and skeleton on the right in the following diagram.
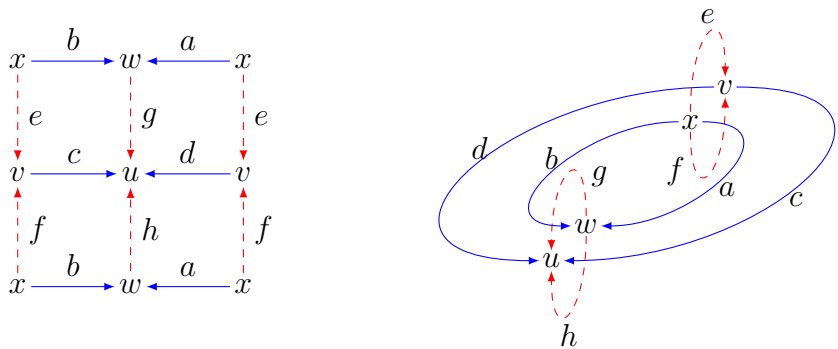
\[
\begin{tikzpicture}[scale=1.5]
    \node at (-2,0) {};
    \node[inner sep=0.5pt, circle] (sw) at (-1,-1) {$x$};
    \node[inner sep=0.5pt, circle] (w) at (-1,0) {$v$};
    \node[inner sep=0.5pt, circle] (nw) at (-1,1) {$x$};
    \node[inner sep=0.5pt, circle] (s) at (0,-1) {$w$};
    \node[inner sep=0.5pt, circle] (m) at (0,0) {$u$};
    \node[inner sep=0.5pt, circle] (n) at (0,1) {$w$};
    \node[inner sep=0.5pt, circle] (se) at (1,-1) {$x$};
    \node[inner sep=0.5pt, circle] (e) at (1,0) {$v$};
    \node[inner sep=0.5pt, circle] (ne) at (1,1) {$x$};
    \draw[-latex, blue] (sw)--(s) node[pos=0.5, above, black] {$b$};
    \draw[-latex, blue] (se)--(s) node[pos=0.5, above, black] {$a$};
    \draw[-latex, blue] (w)--(m) node[pos=0.5, above, black] {$c$};
    \draw[-latex, blue] (e)--(m) node[pos=0.5, above, black] {$d$};
    \draw[-latex, blue] (nw)--(n) node[pos=0.5, above, black] {$b$};
    \draw[-latex, blue] (ne)--(n) node[pos=0.5, above, black] {$a$};
    \draw[-latex, red, dashed] (sw)--(w) node[pos=0.5, right, black] {$f$};
    \draw[-latex, red, dashed] (nw)--(w) node[pos=0.5, right, black] {$e$};
    \draw[-latex, red, dashed] (s)--(m) node[pos=0.5, right, black] {$h$};
    \draw[-latex, red, dashed] (n)--(m) node[pos=0.5, right, black] {$g$};
    \draw[-latex, red, dashed] (se)--(e) node[pos=0.5, right, black] {$f$};
    \draw[-latex, red, dashed] (ne)--(e) node[pos=0.5, right, black] {$e$};
\begin{scope}[xshift=4.5cm]
    \node[inner sep= 1pt] (002) at (0,0,2) {$u$};
    \node[inner sep= 1pt] (001) at (0,0,1.2) {$w$};
    \node[inner sep= 1pt] (00-1) at (0,0,-1.2) {$x$};
    \node[inner sep= 1pt] (00-2) at (0,0,-2) {$v$};
    \draw[-latex, red, dashed] (001) .. controls +(0,0.8,0) and +(0,0.8,0) .. (002) node[pos=0.25, anchor=west] {\color{black}$g$};
    \draw[-latex, red, dashed] (001) .. controls +(0,-0.8,0) and +(0,-0.8,0) .. (002) node[pos=0.5, anchor=north] {\color{black}$h$};
    \draw[-latex, red, dashed] (00-1) .. controls +(0,0.8,0) and +(0,0.8,0) .. (00-2) node[pos=0.5, anchor=south] {\color{black}$e$};
    \draw[-latex, red, dashed] (00-1) .. controls +(0,-0.8,0) and +(0,-0.8,0) .. (00-2) node[pos=0.25, anchor=east] {\color{black}$f$};
    \draw[-latex, blue] (00-1) .. controls +(1,0,0) and +(1,0,0) .. (001) node[pos=0.6, anchor=west] {\color{black}$a$};
    \draw[-latex, blue] (00-1) .. controls +(-1,0,0) and +(-1,0,0) .. (001) node[pos=0.4, anchor=east] {\color{black}$b$};
    \draw[-latex, blue] (00-2) .. controls +(2,0,0) and +(2,0,0) .. (002) node[pos=0.6, anchor=west] {\color{black}$c$};
    \draw[-latex, blue] (00-2) .. controls +(-2,0,0) and +(-2,0,0) .. (002) node[pos=0.4, anchor=east] {\color{black}$d$};
\end{scope}
\end{tikzpicture}
\]
Let $\Lambda$ be the $1$-graph with two vertices connected by two parallel edges used in
Example~\ref{ex:depends on factorisations}; we observed in the same example that the homology of
$\Lambda$ is that of the circle. Then $\Sigma$ is isomorphic to $\Lambda \times \Lambda$, so by the
K\"unneth theorem it has the homology of the 2-torus as in Example~\ref{ex:depends on
factorisations}.
\end{example}

\begin{example}\label{ex:proj plane}
We thank Mike Whittaker for his contributions to the construction and analysis of  this example.

Let $\Lambda$ be the $2$-graph with planar diagram (see Remark~\ref{rmk:planar diagrams}) on the left and skeleton on the
right in the following diagram. As above, the Greek letters in the centres of squares denote the morphisms in $\Lambda^{(1,1)}$
--- so $\alpha = ga = ce$ etc.
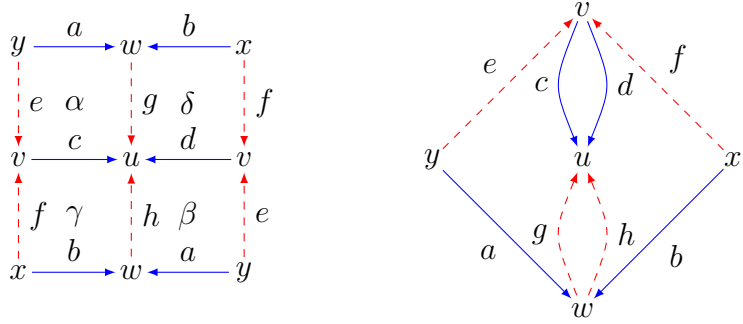
\[
\begin{tikzpicture}
\begin{scope}[scale=1.5]
    \node[inner sep=0.5pt, circle] (sw) at (-1,-1) {$x$};
    \node[inner sep=0.5pt, circle] (w) at (-1,0) {$v$};
    \node[inner sep=0.5pt, circle] (nw) at (-1,1) {$y$};
    \node[inner sep=0.5pt, circle] (s) at (0,-1) {$w$};
    \node[inner sep=0.5pt, circle] (m) at (0,0) {$u$};
    \node[inner sep=0.5pt, circle] (n) at (0,1) {$w$};
    \node[inner sep=0.5pt, circle] (se) at (1,-1) {$y$};
    \node[inner sep=0.5pt, circle] (e) at (1,0) {$v$};
    \node[inner sep=0.5pt, circle] (ne) at (1,1) {$x$};
    \draw[-latex, blue] (sw)--(s) node[pos=0.5, above, black] {$b$};
    \draw[-latex, blue] (se)--(s) node[pos=0.5, above, black] {$a$};
    \draw[-latex, blue] (w)--(m) node[pos=0.5, above, black] {$c$};
    \draw[-latex, blue] (e)--(m) node[pos=0.5, above, black] {$d$};
    \draw[-latex, blue] (nw)--(n) node[pos=0.5, above, black] {$a$};
    \draw[-latex, blue] (ne)--(n) node[pos=0.5, above, black] {$b$};
    \draw[-latex, red, dashed] (sw)--(w) node[pos=0.5, right, black] {$f$};
    \draw[-latex, red, dashed] (nw)--(w) node[pos=0.5, right, black] {$e$};
    \draw[-latex, red, dashed] (s)--(m) node[pos=0.5, right, black] {$h$};
    \draw[-latex, red, dashed] (n)--(m) node[pos=0.5, right, black] {$g$};
    \draw[-latex, red, dashed] (se)--(e) node[pos=0.5, right, black] {$e$};
    \draw[-latex, red, dashed] (ne)--(e) node[pos=0.5, right, black] {$f$};
    \node at (-0.5, -0.5) {$\gamma$};
    \node at (-0.5, 0.5) {$\alpha$};
    \node at (0.5, -0.5) {$\beta$};
    \node at (0.5, 0.5) {$\delta$};
\end{scope}
\begin{scope}[scale=2, xshift=3cm]
    \node[inner sep=.5pt, circle] (u) at (0,0) {$u$};
    \node[inner sep=.5pt, circle] (v) at (0,1) {$v$};
    \node[inner sep=.5pt, circle] (w) at (0,-1) {$w$};
    \node[inner sep=.5pt, circle] (x) at (1,0) {$x$};
    \node[inner sep=.5pt, circle] (y) at (-1,0) {$y$};
    \draw[-latex, blue] (v) .. controls +(-0.2,-0.5) .. (u) node[pos=0.5, left, black] {$c$};
    \draw[-latex, blue] (v) .. controls +(0.2,-0.5) .. (u) node[pos=0.5, right, black] {$d$};
    \draw[-latex, red, dashed] (w) .. controls +(-0.2,0.5) .. (u) node[pos=0.5, left, black] {$g$};
    \draw[-latex, red, dashed] (w) .. controls +(0.2,0.5) .. (u) node[pos=0.5, right, black] {$h$};
    \draw[-latex, blue] (x)--(w) node[pos=0.5, anchor=north west, black] {$b$};
    \draw[-latex, red, dashed] (x)--(v) node[pos=0.5, anchor=south west, black] {$f$};
    \draw[-latex, blue] (y)--(w) node[pos=0.5, anchor=north east, black] {$a$};
    \draw[-latex, red, dashed] (y)--(v) node[pos=0.5, anchor=south east, black] {$e$};
\end{scope}
\end{tikzpicture}
\]

We claim that $H_0(\Lambda) \cong \ZZ$, $H_1(\Lambda) \cong \ZZ/2\ZZ$ and $H_2(\Lambda) = \{0\}$.
Indeed, $C_0(\Lambda) = \ZZ\{u,v,w,x,y\}$, $C_1(\Lambda) = \ZZ\{a,b,c,d,e,f,g,h\}$ and $C_2(\Lambda) =
\ZZ\{\alpha,\beta,\gamma,\delta\}$. Since $\Lambda$ is connected, $H_0(\Lambda)  \cong \ZZ$ which
implies that $\partial_1(C_1(\Lambda))$ has rank 4. Since $\rank C_1(\Lambda) = 8$,
$\rank\ker(\partial_1) = 4$ also. If $\partial_2(n_1\alpha + n_2\beta + n_3\gamma + n_4\delta) =
0$, then consideration of the coefficients of $c$ and $h$ forces $n_1 = -n_3 = n_2$, and then that
the coefficient of $a$ is zero forces $n_1 = n_2 = 0$, and hence $n_3 = 0$ also. Now considering
the coefficient of $d$ shows that $n_4 = - n_2 = 0$. So $\partial_2$ is injective, forcing
$H_2(\Lambda) = \{0\}$, and also that $\rank \partial_2(C_2(\Lambda)) = 4$.  We observed above that
$\rank \ker(\partial_1) = 4$; hence
\[
\rank(H_1(\Lambda)) = \rank(\ker(\partial_1)) - \rank(\partial_2(C_2(\Lambda))) = 0.
\]
It is routine to check that $\{c-d, g-h, c + f - b - h, d + e - a - h\}$ is a basis for $\ker(\partial_1)$. To determine the image of
$\partial_2$, first note that $c + f - b - h = \partial_2(\gamma)$ and $d + e - a - h = \partial_2(\beta)$. Moreover $(c - d) +
(g - h)$ is the image of $\gamma - \delta$, which implies that $H_1(\Lambda)$ is generated by the class of $c - d$. Finally,
$2(c-d) =
\partial_2(\alpha - \beta + \gamma - \delta)$, and since $\{\alpha,\beta,\gamma, \alpha - \beta +
\gamma - \delta)$ is a basis for $C_2(\Lambda)$, it follows that $H_1(\Lambda)  \cong \ZZ/2\ZZ$ as
required.

These homology groups are the same as those of the projective plane.
\end{example}

\begin{example}\label{ex:Klein}
Consider the $2$-graph $\Lambda$ with planar diagram (see Remark~\ref{rmk:planar diagrams}) on the
left and skeleton on the right in the following diagram.
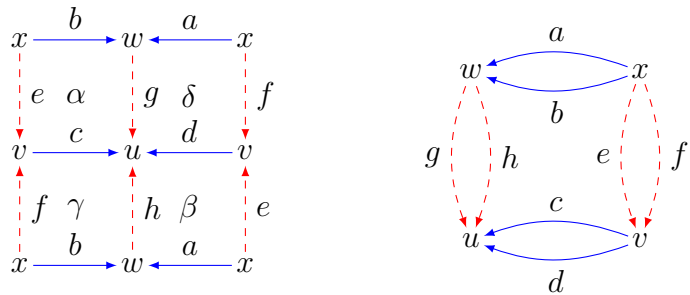
\[
\begin{tikzpicture}[scale=1.5]
    \node[inner sep=0.5pt, circle] (sw) at (-1,-1) {$x$};
    \node[inner sep=0.5pt, circle] (w) at (-1,0) {$v$};
    \node[inner sep=0.5pt, circle] (nw) at (-1,1) {$x$};
    \node[inner sep=0.5pt, circle] (s) at (0,-1) {$w$};
    \node[inner sep=0.5pt, circle] (m) at (0,0) {$u$};
    \node[inner sep=0.5pt, circle] (n) at (0,1) {$w$};
    \node[inner sep=0.5pt, circle] (se) at (1,-1) {$x$};
    \node[inner sep=0.5pt, circle] (e) at (1,0) {$v$};
    \node[inner sep=0.5pt, circle] (ne) at (1,1) {$x$};
    \draw[-latex, blue] (sw)--(s) node[pos=0.5, above, black] {$b$};
    \draw[-latex, blue] (se)--(s) node[pos=0.5, above, black] {$a$};
    \draw[-latex, blue] (w)--(m) node[pos=0.5, above, black] {$c$};
    \draw[-latex, blue] (e)--(m) node[pos=0.5, above, black] {$d$};
    \draw[-latex, blue] (nw)--(n) node[pos=0.5, above, black] {$b$};
    \draw[-latex, blue] (ne)--(n) node[pos=0.5, above, black] {$a$};
    \draw[-latex, red, dashed] (sw)--(w) node[pos=0.5, right, black] {$f$};
    \draw[-latex, red, dashed] (nw)--(w) node[pos=0.5, right, black] {$e$};
    \draw[-latex, red, dashed] (s)--(m) node[pos=0.5, right, black] {$h$};
    \draw[-latex, red, dashed] (n)--(m) node[pos=0.5, right, black] {$g$};
    \draw[-latex, red, dashed] (se)--(e) node[pos=0.5, right, black] {$e$};
    \draw[-latex, red, dashed] (ne)--(e) node[pos=0.5, right, black] {$f$};
    \node at (-0.5, -0.5) {$\gamma$};
    \node at (-0.5, 0.5) {$\alpha$};
    \node at (0.5, -0.5) {$\beta$};
    \node at (0.5, 0.5) {$\delta$};
\begin{scope}[xshift=3cm, yshift=-0.78cm, scale=1.5]
    \node[inner sep=0.5pt, circle] (u) at (0,0) {$u$};
    \node[inner sep=0.5pt, circle] (v) at (1,0) {$v$};
    \node[inner sep=0.5pt, circle] (w) at (0,1) {$w$};
    \node[inner sep=0.5pt, circle] (x) at (1,1) {$x$};
    \draw[blue,-latex,out=160, in=20] (v) to node[pos=0.5,above, black] {$c$} (u);
    \draw[blue,-latex,out=200, in=340] (v) to node[pos=0.5,below, black] {$d$} (u);
    \draw[blue,-latex,out=160, in=20] (x) to node[pos=0.5,above, black] {$a$} (w);
    \draw[blue,-latex,out=200, in=340] (x) to node[pos=0.5,below, black] {$b$} (w);
    \draw[red, dashed, -latex, out=290, in=70] (w) to node[pos=0.5,right, black] {$h$} (u);
    \draw[red, dashed, -latex, out=250, in=110] (w) to node[pos=0.5,left, black] {$g$} (u);
    \draw[red, dashed, -latex, out=290, in=70] (x) to node[pos=0.5,right, black] {$f$} (v);
    \draw[red, dashed, -latex, out=250, in=110] (x) to node[pos=0.5,left, black] {$e$} (v);
\end{scope}
\end{tikzpicture}
\]
One can check, by calculating with bare hands, that the homology of this $2$-graph is the same as
that of the $2$-graph $\Lambda \times_\alpha \ZZ$ of Example~\ref{ex:torsion} with $n = 2$; that
is, the same homology as the Klein bottle. Alternatively, one can deduce this from the topological
realisation (see Remark~\ref{rmk:top realisations} below).
\end{example}

\begin{example}\label{ex:projective covering}
In Example~\ref{ex:proj plane}, we realised the homology of the projective plane using a $2$-graph $\Lambda$. This suggests
that there ought to be a $2$-graph with the homology of the sphere carrying a free action of $\ZZ/2\ZZ$ such that the quotient
is isomorphic to $\Lambda$. By \cite[Remark~5.6]{KP2000} (see also \cite{PaskQuiggEtAl:JA05}), such a $2$-graph must be
a skew product of $\Lambda$ by a functor taking values in $\ZZ/2\ZZ$. Here we present such an example. There is a functor $c
: \Lambda \to \ZZ/2\ZZ$ determined by $c^{-1}(0) = \{b, c, g\}$ and $c^{-1}(1) = \{a, d, e, f, h\}$, and the skew-product
graph $\Lambda \times_c \ZZ/2\ZZ$ has the desired property. The visual intuition that has pervaded this section appears again:
one can check without too much difficulty that the skeleton of $\Lambda \times_c (\ZZ/2\ZZ)$ can be drawn as follows (we
have not labeled the edges since their labels can be deduced from the definition of the skew product and the labels of the vertices).
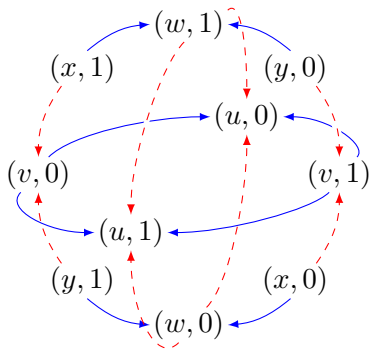
\[
\begin{tikzpicture}[scale=2]
    \node[inner sep= 1pt] (w1) at (xyz spherical cs:longitude=0,latitude=0,radius=1) {\small$(w,1)$};
    \node[inner sep= 1pt] (y0) at (xyz spherical cs:longitude=45,latitude=0,radius=1) {\small$(y,0)$};
    \node[inner sep= 1pt] (v1) at (xyz spherical cs:longitude=90,latitude=0,radius=1) {\small$(v,1)$};
    \node[inner sep= 1pt] (x0) at (xyz spherical cs:longitude=135,latitude=0,radius=1) {\small$(x,0)$};
    \node[inner sep= 1pt] (w0) at (xyz spherical cs:longitude=180,latitude=0,radius=1) {\small$(w,0)$};
    \node[inner sep= 1pt] (y1) at (xyz spherical cs:longitude=225,latitude=0,radius=1) {\small$(y,1)$};
    \node[inner sep= 1pt] (v0) at (xyz spherical cs:longitude=270,latitude=0,radius=1) {\small$(v,0)$};
    \node[inner sep= 1pt] (x1) at (xyz spherical cs:longitude=315,latitude=0,radius=1) {\small$(x,1)$};
    \node[inner sep= 1pt] (u1) at (xyz spherical cs:longitude=0,latitude=90,radius=1) {\small$(u,1)$};
    \node[inner sep= 1pt] (u0) at (xyz spherical cs:longitude=0,latitude=-90,radius=1) {\small$(u,0)$};
    \draw[-latex, blue, out=40, in=185] (x1.75) to (w1.west);
    \draw[-latex, red, dashed, out=235, in=90] (x1.220) to (v0.north);
    \draw[-latex, red, dashed, out=125, in=270] (y1.140) to (v0.south);
    \draw[-latex, blue, out=320, in=175] (y1.285) to (w0.west);
    \draw[-latex, blue, out=220, in=5] (x0.255) to (w0.east);
    \draw[-latex, red, dashed, out=55, in=270] (x0.40) to (v1.south);
    \draw[-latex, blue, out=140, in=355] (y0.105) to (w1.east);
    \draw[-latex, blue] (v0.60) .. controls +(0,0,-0.4) and +(-0.4,0,0) .. (u0.west);
    \draw[-latex, blue] (v0.225) .. controls +(0,0,0.3) and +(-0.3,0,0) .. (u1.west);
    \draw[-latex, blue] (v1.240) .. controls +(0,0,0.3) and +(0.4,0,0) .. (u1.east);
    \draw[-latex, blue] (v1.45) .. controls +(0,0,-0.2) and +(0.3,0,0) .. (u0.east);
    \draw[white, line width=2pt] (w1.255) .. controls +(0,0,0.4) and +(0,0.4,0) .. (u1.north);
    \draw[-latex, red, dashed] (w1.255) .. controls +(0,0,0.4) and +(0,0.4,0) .. (u1.north);
    \draw[-latex, red, dashed] (w0.240) .. controls +(0,0,0.4) and +(0,-0.4,0) .. (u1.south);
    \draw[-latex, red, dashed] (w0.45) .. controls +(0,0,-0.4) and +(0,-0.4,0) .. (u0.south);
    \draw[-latex, red, dashed] (w1.20) .. controls +(0,0,-0.2) and +(0,0.4,0) .. (u0.north);
    \draw[white, line width=3pt, out=305, in=90] (y0.320) to (v1.north);
    \draw[-latex, red, dashed, out=305, in=90] (y0.320) to (v1.north);
\end{tikzpicture}
\]
This picture suggests how to view the action of $\ZZ/2\ZZ$ on the skew-product graph as the action of the antipodal map on
the sphere.

A similar situation arises for the Klein bottle and torus. Let $\Gamma$ denote the crossed product graph $\Lambda
\times_\alpha \ZZ$ of Example~\ref{ex:torsion} with $n = 2$, so that the homology of $\Gamma$ coincides with that of the
Klein bottle. Let $c : \Gamma \to \ZZ/2\ZZ$ be the functor $c(\lambda, n) = n \pmod{2}$. One can check that $\Gamma
\times_c (\ZZ/2\ZZ)$ is isomorphic to $\Lambda \times C_2$ where $\Lambda$ is the $1$-graph from
Example~\ref{ex:torsion} (with $n = 2$), and $C_2$ is the path category of the simple directed cycle of length $2$. In particular,
by the K\"unneth theorem, the homology of $\Gamma \times_c (\ZZ/2\ZZ)$ is isomorphic to that of the torus. So our 2-graph
representative $\Gamma$ of the Klein bottle can be realised as a quotient of a 2-graph representative of a torus by a free
$\ZZ/2\ZZ$ action.
\end{example}

\begin{rmk}\label{rmk:top realisations}
As observed in \cite{kkqs}, the topological realisations of the $2$-graphs of Examples
\ref{ex:sphere}--\ref{ex:Klein} (see Section~\ref{sec:geometric}) are indeed homeomorphic to each
of the sphere, the torus, the projective plane and the Klein bottle as their homology suggests. In
particular, Theorem~\ref{thm:homologies agree} below combined with the descriptions of their topological
realisations in \cite{kkqs} provide an alternative proof that these $2$-graphs have the homology we
have claimed for them.
\end{rmk}

\section{Connection with homology of topological spaces}\label{sec:geometric}

In this section, we show that the homology of the topological realisation $X_\Lambda$ of a $k$-graph as defined in \cite{kkqs}
agrees with the homology of $\Lambda$ defined in \S2. The corresponding fact for a cubical set was known already to Grandis: he
indicates at the end of  \cite[Section~1.8]{Grandis2005} that the result is well known, with a reference to
\cite{Munkres:elements} for the simplicial case. However, we have been unable to locate the details for cubical sets in the literature,
so we include a proof of our result based on that given for simplicial complexes by Hatcher \cite{Hatcher:AlgTop02}. We prove
in Appendix \ref{app:realisations} that the topological realisation of a $k$-graph we define here is homeomorphic to the
topological realisation $\mathcal{R} \widetilde{Q} (\Lambda)$ of the associated cubical set $\widetilde{Q} (\Lambda)$ (see
Appendix \ref{app:grandis}).

In \cite{kkqs}, the topological realisation of a $k$-graph $\Lambda$ is defined as follows. For $n \in \NN^k$, let $[0,n] := \{t
\in \RR^k : 0 \le t \le n\}$. For $t \in \RR^k$,  let $\lfloor t\rfloor$ be the element of $\ZZ^k$ such that $\lfloor t \rfloor_i
= \lfloor t_i \rfloor = \max\{ n \in \ZZ :  n \le t_i \}$ for all $i \le k$. Similarly, define $\lceil t \rceil$ by $\lceil t\rceil_i =
\min\{n \in \ZZ : t_i \le n\}$ for $i \le k$. Consider the following equivalence relation on $\bigsqcup_{\lambda \in
\Lambda} \big(\{\lambda\} \times [0, d(\lambda)]\big)$: for $\mu,\nu \in \Lambda$ and $s,t \in \RR^k$ with $0 \le s \le
d(\mu)$ and $0 \le t \le d(\nu)$, we define
\begin{equation}\label{eq:sim def}
(\mu,s) \sim (\nu,t)\quad \iff \quad s - \lfloor s \rfloor = t - \lfloor t \rfloor
                                \text{ and }\mu(\lfloor s \rfloor, \lceil s \rceil) = \nu(\lfloor t \rfloor, \lceil t \rceil).
\end{equation}
The topological realisation $X_\Lambda$ is the quotient space $\Big(\bigsqcup_{\lambda \in \Lambda} \{\lambda\} \times
[0, d(\lambda)]\Big) / \sim$.  As in \cite{kkqs} we let $[\lambda, t]$ denote the equivalence class of the point $(\lambda, t)$.

\begin{dfn}\label{def:phidef}
For $r \in \NN$, let $\bI^r$ denote the unit cube $[0,1]^r$ in $\RR^r$. Fix an $r$-cube $\lambda \in Q_r (\Lambda)$.
Express $d(\lambda) = e_{i_1} + \cdots + e_{i_r}$ where $i_1  < \cdots < i_r$. Let $\iota_\lambda: \bI^r \to
X_\Lambda$ denote the map $(t_1, \dots, t_r) \mapsto \big[\lambda, \sum^r_{m=1} t_m e_{i_m}\big]$. Then
$\Phi(\lambda) := \iota_\lambda$ defines a homomorphism $\Phi : C_r (\Lambda) \to C_r^{\Top}(X_\Lambda)$.
\end{dfn}

\begin{rmk} \label{rmk:phidef}
The map $\Phi$ intertwines the boundary maps, so is a chain map. It therefore induces a homomorphism $\Phi_* : H_*
(\Lambda) \to H_*^{\Top} (X_\Lambda)$.

It will be shown in \cite{kkqs} that each $k$-graph morphism $\theta : \Lambda \to \Gamma$ induces
a continuous map $\widetilde{\theta} : X_\Lambda \to X_\Gamma$ such that $\widetilde{\theta} \circ
\iota_\lambda = \iota_{\theta(\lambda)}$ for all $\lambda \in Q(\Lambda)$. Hence both the chain map
$\Phi$ and the homomorphism $\Phi_*$ of homology are natural in $\Lambda$. (with respect to
$k$-graph morphisms).
\end{rmk}

\begin{thm}\label{thm:homologies agree}
Let $\Lambda$ be a $k$-graph. For each $r \ge 0$, the map $\Phi_* : H_r(\Lambda) \to
H_r^{\Top}(X_\Lambda)$ is an isomorphism. Moreover this isomorphism is natural in $\Lambda$.
\end{thm}

Our proof parallels the argument of the first three paragraphs of \cite[Theorem~2.27]{Hatcher:AlgTop02} where it is shown that
the singular homology of a $\Delta$-complex (see \cite[page~103]{Hatcher:AlgTop02}) is the same as its simplicial homology.
We first need to do some setting up.

\begin{rmk}\label{rmk:Massey vs Hatcher}
We claim that Massey's definition of singular homology, which is based on cubes, is equivalent to the usual one based on simplices.
By the uniqueness theorem of \cite{Milnor:PJM62}, if $X$ has the homotopy type of a CW-complex, then any homology theory
on $X$ which satisfies the Eilenberg-Steenrod axioms \cite{EilenbergSteenrod:FoundationsAlgTop} and which is additive in the
sense that it carries disjoint unions to direct sums is naturally isomorphic to the usual singular homology. The Eilenberg-Steenrod
axioms and additivity are all verified for Massey's singular homology in \cite[Chapter~VII]{Massey}: Axiom~1 is (3.4), Axiom~2
is (3.5), Axiom 3 is (7.6.1), Axiom~4 is Theorem~5.1, Axiom~5 is Theorem~6.1, Axiom~6 is Theorem~6.2, Axiom~7 is
Example~2.1, and additivity is Proposition~2.7. Alternatively that Massey's homology agrees with the simplicial formulation also
follows from the original uniqueness theorem \cite[Theorem~10.1]{EilenbergSteenrod:FoundationsAlgTop} since we can
triangulate $X_\Lambda$ by adding a vertex at the centre of each cube (thereby dividing each $r$-cube into $2^r r!$
$r$-simplices).
\end{rmk}

To run Hatcher's argument, we use the cellular structure of $X_\Lambda$ regarded as a CW-complex. For $0 \le m \le k$ let
$X_m$ denote the union of the images of the $\iota_\lambda$ where $\lambda$ ranges over all $r$-cubes with $r \le m$.   We
formally define $C_r^\Lambda(X_m) = C_r(\Lambda)$ if $m \ge r$ and to be zero otherwise.  We obtain a nested sequence
\[
    C_*^\Lambda(X_0) \subseteq C_*^\Lambda(X_1) \subseteq \cdots \subseteq C_*^\Lambda(X_k) = C_*(\Lambda)
\]
of complexes. In particular, for $l \le m$ we may form the quotient complex
\[
    C_*^\Lambda(X_m, X_{l}) := C_*^\Lambda(X_m)/ C_*^\Lambda(X_l),
\]
which has relative homology groups $H_*^\Lambda(X_m, X_{l})$.  Then
\begin{equation} \label{eq:relhom}
H_r^\Lambda(X_m, X_{m-1}) \cong C_r^\Lambda(X_m, X_{m-1}) =
\begin{cases}
C_r (\Lambda) & \text{ if } m = r,\\
\{ 0 \} & \text{ otherwise. }
\end{cases}
\end{equation}

Since every short exact sequence of complexes induces a long exact sequence in homology (see
\cite[Theorem 2.16]{Hatcher:AlgTop02}), we obtain a long exact sequence
\begin{equation}\label{eq:relative LES}
\begin{split}
\cdots \longrightarrow  H^\Lambda_{r+1}(X_m, X_{m-1})
    \longrightarrow H^\Lambda_r&(X_{m-1})
    \longrightarrow H^\Lambda_r(X_m)
    \longrightarrow H^\Lambda_r(X_m, X_{m-1}) \\
    &\longrightarrow H^\Lambda_{r-1}(X_{m-1}) \longrightarrow
    \cdots \longrightarrow H^\Lambda_0(X_m, X_{m-1}).
\end{split}
\end{equation}

The map $\Phi : C_*(\Lambda) \to C_*^{\Top}(X_\Lambda)$ induces a map from $C_*(X_m)$ to
$C_*^{\Top}(X_m)$ for each $m$. Hence, it induces a map, also called $\Phi$, from $C^\Lambda_*(X_m,
X_{m-1})$ to $C^{\Top}_*(X_m, X_{m-1})$.

The crucial step in Hatcher's proof of \cite[Theorem 2.27]{Hatcher:AlgTop02} is the following
isomorphism.

\begin{lem} \label{lem:reliso}
With notation as above, the induced map
\[
\Phi_* : H_r^\Lambda(X_m, X_{m-1}) \to H_r^{\Top}(X_m, X_{m-1})
\]
is an isomorphism for each $r,m$.
\end{lem}
\begin{proof}
Suppose that  $r \ne m$.   Then  $H_r^\Lambda(X_m, X_{m-1}) = \{0\}$ by~\eqref{eq:relhom} and $H_r^{\Top}(X_m,
X_{m-1}) = \{0\}$ by \cite[Lemma 2.3.4\,(a)]{Hatcher:AlgTop02}. Hence $\Phi_* : H_r^\Lambda(X_m, X_{m-1}) \to
H_r^{\Top}(X_m, X_{m-1})$ is an isomorphism for $m \ne r$. Since
\[
H_r^\Lambda(X_r, X_{r-1}) \cong C_r (\Lambda) = \ZZ{Q_r}(\Lambda)  \cong
H_r^{\Top}(Q_r (\Lambda)  \times \bI^r , Q_r (\Lambda)  \times \partial \bI^r),
\]
it suffices to show that the canonical map $Q_r(\Lambda) \times \bI^r \to X_r$ given by $(\lambda, t) \mapsto
\iota_\lambda(t)$  induces an isomorphism
\[
H_r^{\Top}(Q_r (\Lambda)  \times \bI^r , Q_r (\Lambda)  \times \partial \bI^r) \cong H_r^{\Top}(X_r, X_{r-1}).
\]
To see this, observe that $(X_r, X_{r-1})$ is a good pair (see \cite[p.~114]{Hatcher:AlgTop02})  in the sense that $X_{r-1}$ is
a nonempty closed subset of $X_r$ which is a deformation retract of the open set
\[
X_{r-1} \cup \{[\lambda,t] : \lambda \in Q_r(\Lambda),
\min\{t_i, 1-t_i\} < 1/3\text{ for } 1 \le i \le r \}.
\]
Let $X_r/X_{r-1}$ be the quotient of $X_r$ obtained by identifying $X_{r-1}$ to a point. That  $(X_r, X_{r-1})$ is a good
pair combines with \cite[Proposition~2.22]{Hatcher:AlgTop02} and Remark~\ref{rmk:Massey vs Hatcher} to show that
\[
 H_r^{\Top}(X_r, X_{r-1}) \cong H_r^{\Top}(X_r/X_{r-1}).
\]
Moreover, $\Phi_r$ induces a homeomorphism of $(Q_r(\Lambda)  \times \bI^r) / (Q_r(\Lambda) \times
\partial \bI^n)$ with $X_r/X_{r-1}$. Since $(Q_r(\Lambda)  \times \bI^r, Q_r(\Lambda)  \times \partial \bI^r)$ is also a
good pair, the result follows from another application of \cite[Proposition~2.22]{Hatcher:AlgTop02}.
\end{proof}

\begin{proof}[Proof of Theorem~\ref{thm:homologies agree}]
The naturality of $\Phi_*$ was observed in Remark~\ref{rmk:phidef}. So we just need to
show that $\Phi_*$ is an isomorphism.

Both $H_r(\Lambda)$ and $H_r(X_\Lambda)$ are trivial for $r > k$, so we may assume that $0 \le r \le k$. Fix $m \in \NN$.
If $r \le m$ then we may regard the map $\Phi : C_r(\Lambda) \to C_r^{\Top}(X_\Lambda)$ given in
Definition~\ref{def:phidef} as a map from $C^\Lambda_r(X_m)$ to $C^{\Top}_r(X_m)$; whereas if $r > m$ then both
$C^\Lambda_r(X_m)$ and $C^{\Top}_r(X_m)$ are trivial, and we define $\Phi : C^\Lambda_r(X_m) \to
C^{\Top}_r(X_m)$ to be the trivial map between trivial groups. As in Remark~\ref{rmk:phidef}, $\Phi$ intertwines the
boundary maps, and so induces a homomorphism $\Phi_* : H^\Lambda_*(X_m) \to H^{\Top}_*(X_m)$.

We claim that these maps are all isomorphisms. We proceed by induction on $m$. Our base case is $m
= 0$. Since $X_0$ is equal to the discrete space $\Lambda^0$, each of $H^\Lambda_0(X_{0})$ and
$H^{\Top}_0(X_{0})$ is canonically isomorphic to $\ZZ\Lambda^0$, and $\Phi_*$ is the identity map.
Moreover, for $r \ge 1$, we have $H^\Lambda_r(X_{0})=H^{\Top}_r(X_{0}) =\{0\}$, so $\Phi_*$ is
trivially an isomorphism. Now fix $m \ge 1$ and suppose as an inductive hypothesis that $\Phi_*$ is
an isomorphism between $H_*^\Lambda(X_{m-1})$ and $H^{\Top}_*(X_{m-1})$.  Fix  $r \ge 0$.  Since
$\Phi_*$ induces a map of short exact sequences of complexes, the naturality of the connecting map
in the long exact sequence arising from a short exact sequence of complexes yields the following
commuting diagram.
\[
\begin{tikzpicture}[yscale=1.5, xscale=3]
    \node[inner sep=1pt] (01) at (0,1) {$H^\Lambda_r(X_m, X_{m-1})$};
    \node[inner sep=1pt] (11) at (1.1,1) {$H^\Lambda_r(X_{m-1})$};
    \node[inner sep=1pt] (21) at (2,1) {$H^\Lambda_r(X_m)$};
    \node[inner sep=1pt] (31) at (3,1) {$H^\Lambda_{r-1}(X_m, X_{m-1})$};
    \node[inner sep=1pt] (41) at (4.1,1) {$H^\Lambda_{r-1}(X_{m-1})$};
    \node[inner sep=1pt] (00) at (0,0) {$H^{\Top}_r(X_m, X_{m-1})$};
    \node[inner sep=1pt] (10) at (1.1,0) {$H^{\Top}_r(X_{m-1})$};
    \node[inner sep=1pt] (20) at (2,0) {$H^{\Top}_r(X_m)$};
    \node[inner sep=1pt] (30) at (3,0) {$H^{\Top}_{r-1}(X_m, X_{m-1})$};
    \node[inner sep=1pt] (40) at (4.1,0) {$H^{\Top}_{r-1}(X_{m-1})$};
    \foreach \x/\xx in {0/1,1/2,2/3,3/4,4/5} {
        \ifnum \x < 4 {
            \draw[-latex] (\x1)--(\xx1);
            \draw[-latex] (\x0)--(\xx0);
        }\fi
        \draw[-latex] (\x1)--(\x0) node[pos=0.5, anchor=west, inner sep=1pt] {$\Phi_*$};
    }
\end{tikzpicture}
\]
The inductive hypothesis ensures that the second and fifth vertical maps are isomorphisms, and the first and fourth maps are
isomorphisms by Lemma~\ref{lem:reliso}. Thus the Five Lemma (see, for example, \cite[p\,129]{Hatcher:AlgTop02}) implies
that the middle vertical map is also an isomorphism, completing the induction. Hence, $\Phi_*: H^\Lambda_r(X_m) \to
H^{\Top}_r(X_m)$ is an isomorphism for all $m$. Since $H_r(\Lambda) = H^\Lambda_r(X_k)$ and $H^{\Top}_r(X) =
H^{\Top}_r(X_k)$ for all $r \ge 0$ the desired result follows.
\end{proof}

\section{Cohomology and twisted $k$-graph $C^*$-algebras}\label{sec:cohomology}

In this section we introduce cohomology for $k$-graphs and indicate how a $\TT$-valued $2$-cocycle
may be used to twist a $k$-graph $C^*$-algebra. We first define the cohomology of a $k$-graph and
provide a Universal Coefficient Theorem. We then show how to associate to each $\TT$-valued
$2$-cocycle $\phi$ on $\Lambda$ a twisted $C^*$-algebra $C^*_\phi(\Lambda)$. We obtain as
relatively elementary examples all noncommutative tori and the Heegaard-type quantum 3-spheres of
\cite{BaumHajacEtAl:K-th05}. We will study cohomology for $k$-graphs and the structure of twisted
$k$-graph $C^*$-algebras in greater detail  in \cite{KPS4}.

\begin{ntn}
Let $\Lambda$ be a $k$-graph and let $A$ be an abelian group. For $r \in \NN$, we write
$\Ccub{r}(\Lambda, A)$ for the collection of all functions $f : Q_r(\Lambda) \to A$. We identify
$\Ccub{r}(\Lambda, A)$ with $\Hom(C_r(\Lambda), A)$ in the usual way. Define maps $\dcub{r} :
\Ccub{r}(\Lambda,A) \to \Ccub{r+1}(\Lambda,A)$ by
\[
\dcub{r}(f)(\lambda) := f(\partial_{r+1}(\lambda)) =
\sum^{r+1}_{i=1}\sum_{l=0}^1 (-1)^{i+l} f(F_i^l(\lambda)).
\]
Then $( \Ccub{*}(\Lambda, A), \dcub{*})$ is a cochain complex.
\end{ntn}

Mac Lane \cite[Chapter~II, Equation~(3.1)]{MacLane} associates a cochain complex to a chain complex
and an abelian group in a similar way, but with a slightly different sign convention for the
boundary map. The resulting cohomology is isomorphic to the following.

\begin{dfn}
We define the \emph{cohomology} $\Hcub*(\Lambda, A)$ of the $k$-graph $\Lambda$ with coefficients
in $A$ to be the cohomology of the complex $\Ccub*(\Lambda,A)$; that is $\Hcub{r}(\Lambda,A) :=
\ker(\dcub{r})/\image(\dcub{r-1})$. For $r \ge 0$, we write $\Zcub{r}(\Lambda,A) := \ker(\dcub{r})$
for the group of $r$-cocycles, and for $r > 0$, we write $\Bcub{r}(\Lambda,A) = \image(\dcub{r-1})$
for the group of $r$-coboundaries.
\end{dfn}

\begin{thm}[Universal Coefficient Theorem]\label{thm:uct}
Let $\Lambda$ be a $k$-graph, and let $A$ be an abelian group. For each $r \ge 0$, there is a short
exact sequence
\[
    0 \longrightarrow \Ext(H_{r-1}(\Lambda), A)
        \xrightarrow{\alpha} \Hcub{r}(\Lambda, A)
        \xrightarrow{\beta} \Hom(H_r(\Lambda), A)
        \longrightarrow 0,
\]
and the maps $\alpha$ and $\beta$ are natural in $A$ and $\Lambda$.
\end{thm}
\begin{proof}
This follows directly from Mac Lane's theorem \cite[Theorem~III.4.1]{MacLane} applied to the
complex $C_*(\Lambda)$.
\end{proof}

Recall from \cite{KP2000} that a $k$-graph $\Lambda$ is \emph{row-finite} if $v\Lambda^n$ is finite
for all $v \in \Lambda^0$ and $n \in \NN^k$, and is \emph{locally convex} if, whenever $1 \le i
\not= j \le k$ and $\lambda \in \Lambda^{e_i}$ with $r(\lambda) \Lambda^{e_j} \not= \emptyset$, we
have $s(\lambda)\Lambda^{e_j} \not= \emptyset$ also.

We will follow the usual convention of writing the binary operation in an abelian group $A$ additively, except when $A = \TT$
where it is written multiplicatively.

\begin{dfn}[{cf. \cite[Equation~(3.1)]{RSY1} and \cite[Theorem~C.1(i)--(ii)]{RSY2}}]\label{def:twisted CK}
Let $\Lambda$ be a row-finite locally convex $k$-graph and fix $\phi \in \Zcub2(\Lambda,\TT)$. A Cuntz-Krieger $\phi$-representation of $\Lambda$ in a $C^*$-algebra $A$ is a set $\{p_v : v \in
\Lambda^0\} \subseteq A$ of mutually orthogonal projections and a set $\{s_\lambda : \lambda \in
\bigcup^k_{i=1} \Lambda^{e_i}\} \subseteq A$ satisfying
\begin{enumerate}
    \item\label{it:new CK1} for all $\lambda \in \Lambda^{e_i}$, $s_\lambda^*s_\lambda =
        p_{s(\lambda)}$;
    \item\label{it:new CK2} for all $1 \le i < j \le k$ and $\mu,\mu' \in \Lambda^{e_i}$,
        $\nu,\nu' \in \Lambda^{e_j}$ such that $\mu\nu = \nu'\mu'$,
    \[
        s_{\nu'} s_{\mu'} = \phi(\mu\nu)s_\mu s_\nu;\text{ and}
    \]
    \item\label{it:new CK3} for all $v \in \Lambda^0$ and all $i = 1, \dots, k$ such that
        $v\Lambda^{e_i} \not= \emptyset$,
    \[
        p_v =  \sum_{\lambda \in v\Lambda^{e_i}} s_\lambda s_\lambda^*.
    \]
\end{enumerate}
\end{dfn}

The condition that a set $\{p_v : v \in \Lambda^0\}$ consists of mutually orthogonal projections is
characterised by the algebraic relations $p_v^* = p_v^2 = p_v$ and $p_v p_w = \delta_{v,w} p_v$ for
all $v,w \in \Lambda^0$. Given any collection $\{p_v : v \in \Lambda^0\}$ in a $^*$-algebra
satisfying these relations, and given any family $\{s_\lambda : \lambda \in \bigcup^k_{i=1}
\Lambda^{e_i}\}$ in the same $^*$-algebra satisfying relation~(\ref{it:new CK1}), the norm of the
image of each $p_v$ and of each $s_\lambda$ under any representation on Hilbert space is at most 1.
So as in \cite[Definition~1.2]{Blackadar:MS85}, there is a universal $C^*$-algebra generated by a
Cuntz-Krieger $\phi$-representation of $\Lambda$. A priori, this could be the zero algebra; but we
will exhibit some interesting examples (see Examples~\ref{ex:irrational},  \ref{ex:noncommutative
tori}, \ref{ex:Heegaard}) where it is not, and we will show in the forthcoming article \cite{KPS4}
that in fact there is always a Cuntz-Krieger $\phi$-representation of $\Lambda$ in which every
generator is nonzero.

\begin{dfn}\label{def:twisted algebra}
Let $\Lambda$ be a row-finite locally convex $k$-graph. Let $\phi \in \Zcub2(\Lambda,\TT)$. We
define $C^*_\phi(\Lambda)$ to be the universal $C^*$-algebra generated by a Cuntz-Krieger
$\phi$-representation of~$\Lambda$.
\end{dfn}

\goodbreak

\begin{prop}\label{prp:cohomologous cocycles}
Let $\Lambda$ be a row-finite locally convex $k$-graph.
\begin{enumerate}
\item\label{it:trivial twist} Let $1$ denote the identity element of $C_2(\Lambda, \TT)$. Then
    $C^*_1(\Lambda)$ is canonically isomorphic to the $k$-graph algebra $C^*(\Lambda)$ defined
    in \cite{RSY1}.
\item\label{it:cohomologous twists} Let $\psi, \phi \in Z^2(\Lambda, \TT)$, and suppose that
    $\alpha \in C^1(\Lambda,\TT)$ satisfies $\phi = \delta^1(\alpha)\psi$ so that $\phi$ and
    $\psi$ are cohomologous. Let $\{p^\psi_v : v \in \Lambda^0\}$, $\{s^\psi_\lambda : \lambda
    \in \bigsqcup^k_{i=1} \Lambda^{e_i}\}$ be the universal generating Cuntz-Krieger
    $\psi$-representation of $\Lambda$ and similarly for $\phi$. Then there is an isomorphism
    $\pi : C^*_\psi(\Lambda) \to C^*_\phi(\Lambda)$ such that $\pi(p^\psi_v) = p^\phi_v$ for
    all $v \in \Lambda^0$ and $\pi(s^\psi_\lambda) = \alpha(\lambda) s^\phi_\lambda$ for all
    $\lambda \in \bigcup^k_{i=1}\Lambda^{e_i}$.
\end{enumerate}
\end{prop}
\begin{proof}
(\ref{it:trivial twist}) The combination of \cite[Theorem~C.1 and Lemma~B.4]{RSY2} shows that
$C^*(\Lambda)$ is the universal $C^*$-algebra generated by elements satisfying the relations of
Definition~\ref{def:twisted CK} with $\phi(\mu\nu) = 1$ for all $\mu\nu \in Q_2(\Lambda)$.

(\ref{it:cohomologous twists}) For $\lambda \in  \bigcup^k_{i=1} \Lambda^{e_i}$, let $t_\lambda :=
\alpha(\lambda) s^\phi_\lambda$. If $\mu\nu = \nu'\mu'$ where $\mu,\mu' \in \Lambda^{e_i}$,
$\nu,\nu' \in \Lambda^{e_j}$ and $1 \le i < j \le k$, then $\delta^1(\alpha) =
\alpha(\mu')^{-1}\alpha(\nu')^{-1}\alpha(\mu)\alpha(\nu)$. Hence
\[
\alpha(\nu') \alpha(\mu') \phi(\mu\nu)
    = \alpha(\nu')\alpha(\mu') \delta^1(\alpha)(\mu\nu) \psi(\mu\nu)
    = \alpha(\mu)\alpha(\nu)\psi(\mu\nu).
\]
Using this, we calculate:
\[
t_{\nu'} t_{\mu'} = \alpha(\nu') \alpha(\mu') s_{\nu'} s_{\mu'}
    = \alpha(\nu')\alpha(\mu') \phi(\mu\nu)s_\mu s_\nu
    = \alpha(\mu)\alpha(\nu)\psi(\mu\nu)s_\mu s_\nu
    = \psi(\mu\nu)t_\mu t_\nu.
\]
So $\{t_\lambda : \lambda \in \bigcup^k_{i=1} \Lambda^{e_i}\}$  satisfies
Definition~\ref{def:twisted CK}(\ref{it:new CK2}) for the cocycle $\psi$. Hence the collections
$\{p^\phi_v : v \in \Lambda^0\}$ and $\{t_\lambda : \lambda \in \bigcup^k_{i=1} \Lambda^{e_i}\}$ in
$C^*_\phi(\Lambda)$ constitute a Cuntz-Krieger $\psi$-representation of $\Lambda$. The universal
property of $C^*_\psi(\Lambda)$ therefore gives a homomorphism $\pi : C^*_\psi(\Lambda) \to
C^*_\phi(\Lambda)$ such that $\pi(p^\psi_v) = p^\phi_v$ for all $v \in \Lambda^0$ and
$\pi(s^\psi_\lambda) = t_\lambda = \alpha(\lambda) s^\phi_\lambda$ for all $\lambda \in
\bigcup^k_{i=1}\Lambda^{e_i}$. Reversing the roles of $\psi$ and $\phi$ in the above calculation
yields an inverse, so $\pi$ is an isomorphism.
\end{proof}

\begin{example}\label{ex:irrational}
Let $T_2$ denote $\NN^2$ regarded as a $2$-graph with degree functor the identity map (see
Examples~\ref{ex:kgraphs}(\ref{it:Tsubk})). Fix $\theta \in [0,1)$. There is precisely one 2-cube in $T_2$, namely $(1, 1)$.
Define $\phi \in \Zcub2(T_2, \TT )$ by $\phi(1,1) = e^{2{\pi}i\theta}$. By definition, $C_\phi^*(T_2)$ is the universal
$C^*$-algebra generated by unitaries $S_{e_1}$ and $S_{e_2}$ satisfying
\[
S_{e_2} S_{e_1} = e^{2{\pi}i\theta} S_{e_1} S_{e_2}.
\]
That is, $C_\phi^*(T_2)$ is the rotation algebra $A_\theta$.
\end{example}

\begin{rmk}
Theorem~2.1 of~\cite{kpq} says that the obstruction to a product system over $\NN^2$ of
$\CC$-correspondences being the product system associated to the $2$-graph $T_2$ is measured by the
element $\omega \in \TT$ which implements the module isomorphism $\CC \otimes \CC \to \CC \otimes
\CC$ between $X_{(1,0)} \otimes X_{(0,1)}$ and $X_{(0,1)} \otimes X_{(1,0)}$. We may regard
$H^2(T_2, \TT)$ as the receptacle for this obstruction.
\end{rmk}

\begin{example}\label{ex:noncommutative tori}
More generally consider the $k$-graph $T_k$ for $k \ge 2$.  Then the twisted $k$-graph $C^*$-algebras over $T_k$ correspond
exactly to the noncommutative tori (see for example \cite{IP2008}, \cite{EL2008}; note that their sign conventions differ). Let
$\theta$ be a skew-symmetric $k \times k$ real matrix, then the associated noncommutative torus $A_\theta$ is the universal
$C^*$-algebra generated by $k$ unitaries $u_1, \dots, u_k$, satisfying (see \cite{IP2008})
\begin{equation}\label{eq:noncommutative tori rels}
 u_n u_m = e^{2{\pi}i\theta_{m, n}} u_m u_n \qquad\textrm{for all}\quad 1 \le m, n \le k.
\end{equation}
Recall that $Q_2(T_k) = \{ e_m + e_n \mid 1 \le m < n \le k \}$. Set $\phi_\theta(e_m + e_n) = e^{2{\pi}i\theta_{m,
n}}$. Then $\phi(\theta)$ is a 2-cocycle. Moreover $C_{\phi(\theta)}^*(T_k)$ is the universal $C^*$-algebra generated by $k$
unitaries $S_{e_1}, \dots, S_{e_k}$ satisfying~\eqref{eq:noncommutative tori rels}. Hence, $A_\theta \cong
C_{\phi(\theta)}^*(T_k)$.
\end{example}

\begin{example}\label{ex:Heegaard}
In \cite{BaumHajacEtAl:K-th05} the authors describe $C^*$-algebras $C(S^3_{pq\theta})$ where $p, q,
\theta$ are parameters in $[0,1)$. They show that $C (S^3_{pq\theta}) \cong
C (S^3_{00\theta})$ \cite[Theorem~2.8]{BaumHajacEtAl:K-th05} for all $p,q,\theta$. By definition, $C (S^3_{00\theta})$
is the universal $C^*$-algebra generated by elements $S$ and $T$ satisfying
\begin{gather}
(1-SS^*)(1-TT^*) = 0, \label{eq:Heegaard1}\\
S^*S = T^*T = 1, \label{eq:Heegaard2}\\
ST = e^{2\pi i\theta}TS,\text{ and} \label{eq:Heegaard3}\\
ST^* = e^{-2\pi i\theta}T^*S.\label{eq:Heegaard4}
\end{gather}
It was shown in \cite[Remark~3.3]{HajacMatthesEtAl:ART06} that $C (S^3_{000})$ is isomorphic to the Cuntz-Krieger
algebra of the unique $2$-graph $\Lambda$ with skeleton $E_\Lambda$ as pictured below.
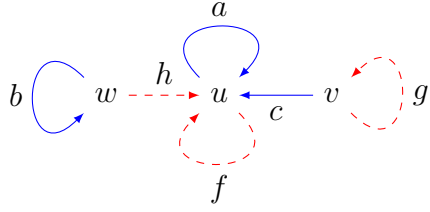
\[
\begin{tikzpicture}[scale=1.5]
    \node (u) at (0,0) {$u$};
    \node (v) at (1,0) {$v$};
    \node (w) at (-1,0) {$w$};
    \draw[blue, -latex, in=130, out=45] (u.north west) .. controls +(-0.6, 0.6) and +(0.6, 0.6) .. (u.north east) node[pos=0.5, above, black] {$a$};
    \draw[blue, -latex] (w.north west) .. controls +(-0.6, 0.6) and +(-0.6, -0.6) .. (w.south west) node[pos=0.5, left, black] {$b$};
    \draw[blue, -latex] (v)--(u) node[pos=0.5, below, black] {$c$};
    \draw[red, dashed, -latex] (u.south east) .. controls +(0.6, -0.6) and +(-0.6, -0.6) .. (u.south west) node[pos=0.5, below, black] {$f$};
    \draw[red, dashed, -latex] (v.south east) ..  controls +(0.6, -0.6) and +(0.6, 0.6) .. (v.north east)  node[pos=0.5, right, black] {$g$};
    \draw[red, dashed, -latex] (w)--(u) node[pos=0.5, above, black] {$h$};
\end{tikzpicture}
\]
Specifically, the isomorphism $C(S^3_{000}) \to C^*(\Lambda)$ carries $S$ to $s_a + s_b + s_c$ and
$T$ to $s_f + s_g + s_h$.

Note that $T_2 = \NN^2$ so the degree map on $\Lambda$ yields a $2$-graph morphism $f : \Lambda \to T_2$.   A routine
computation shows that $f_*$ induces an isomorphism on homology. Hence by Theorem~\ref{thm:uct}, $f^*$ induces an
isomorphism $H^2(T_2, \TT) \cong H^2(\Lambda, \TT)$.

Let $\alpha = ah = hb$, $\beta = cg = fc$ and $\tau = af = fa$; so $Q_2(\Lambda) = \{ \alpha, \beta, \tau \}$. For each
$\theta \in [0,1)$ the $2$-cocycle on $T_2$ determined by $(1,1) \mapsto e^{-2\pi i\theta}$ pulls back to a $2$-cocycle
$\phi_\theta$ on $\Lambda$ satisfying $\phi_\theta(\alpha) = \phi_\theta(\beta) = \phi_\theta(\tau) = e^{-2\pi i\theta}$
(the preceding paragraph shows that every $2$-cocycle on $\Lambda$ is cohomologous to one of this form). Fix $\theta \in
[0,1)$ and let $\{s_\lambda : \lambda \in \bigcup^k_{i=1} \Lambda^{e_i}\}$ and $\{p_v : v \in \Lambda^0\}$ be the
generators of $C^*_{\phi(\theta)}(\Lambda)$. Define $\overline{S}, \overline{T} \in C^*_{\phi(\theta)}(\Lambda)$ by
$\overline{S} := s_a + s_b + s_c$ and $\overline{T} = s_f + s_g + s_g$.   We have
\[
    \overline{S}\overline{T} = s_a s_f + s_c s_g + s_a s_h
        = e^{2\pi i\theta} s_fs_a + e^{2\pi i\theta}s_fs_c + e^{2\pi i\theta}s_hs_b
        = e^{2\pi i\theta} \overline{T}\overline{S}.
\]
So $\overline{S}, \overline{T}$ satisfy~\eqref{eq:Heegaard3}. Moreover
\begin{align*}
\overline{T}^* \overline{S}  &= \overline{T}^* p_u \overline{S} = ( s_f^* + s_g^* + s_h^* ) ( s_\alpha s_\alpha^* + s_\beta s_\beta^* + s_\tau s_\tau^* ) ( s_a + s_b + s_c ) \\
&=  s_f^* ( s_\beta s_\beta^* ) s_c + s_f^* ( s_\tau s_\tau^* ) s_a + s_h^*
( s_\alpha s_\alpha^* ) s_a \\
&= s_f^* ( e^{2\pi i \theta} s_f s_c )( s_g^* s_c^* ) s_c + s_f^* ( e^{2\pi i
\theta} s_f s_a )( s_f^* s_a^* ) s_a + s_h^* ( e^{2\pi i \theta} s_h s_b )(
s_h^* s_a^* ) s_a \\
&= e^{2\pi i \theta} ( s_c s_g^* + s_a s_f^* + s_b s_h^*   )
= e^{2\pi i \theta} ( s_a + s_b + s_c ) ( s_f^* + s_g^* + s_h^* ) \\
&= e^{2\pi i \theta} \overline{S} \overline{T}^* ,
\end{align*}
which establishes~\eqref{eq:Heegaard4}. That $\overline{S},\overline{T}$ also satisfy
\eqref{eq:Heegaard1}~and~\eqref{eq:Heegaard2} is routine.  Hence by the universal property of $C(S^3_{00\theta})$ the map
$S \to \overline{S}$ and $T$ to $\overline{T}$ extends to a homomorphism $\rho$ from $C(S^3_{00\theta})$ to
$C^*_{\phi(\theta)}(\Lambda)$.

Now let $S$ and $T$ be the generators of $C(S^3_{00\theta})$. Define
\[
    q_w = 1 - SS^*,\quad q_v = 1 - TT^*,\quad\text{ and }\quad q_u = SS^*TT^*,
\]
and
\[
t_\eta = q_{r(\eta)} S q_{s(\eta)} \text{ for $\eta \in \Lambda^{e_1}$},\quad\text{ and }
t_\eta = q_{r(\eta)} T q_{s(\eta)} \text{ for $\eta \in \Lambda^{e_2}$}.
\]
It is routine to check that the pair $\{q_u, q_v, q_w\}$, $\{t_a, t_b, t_c, t_f, t_g, t_h\}$ is a Cuntz-Krieger $\phi
(\theta)$-representation of $\Lambda$ in $C(S^3_{00\theta})$. So the universal property of $C^*_{\phi(\theta)}(\Lambda)$
yields a homomorphism $\psi : C^*_{\phi(\theta)} ( \Lambda )\to C ( S^3_{00\theta} )$ such that $\psi(p_x) = q_x$ for $x
\in \Lambda^0$ and $\psi(s_\eta) = t_\eta$ for $\eta \in \Lambda^{e_1} \cup \Lambda^{e_2}$. One verifies that $\psi =
\rho^{-1}$ and it follows that $C^*_{\phi(\theta)} ( \Lambda ) \cong C ( S^3_{00\theta} )$.

Our analysis of $\Hcub2(\Lambda , \TT )$, together with Proposition~\ref{prp:cohomologous
cocycles}, therefore shows that the collection of twisted $2$-graph $C^*$-algebras associated to
$\Lambda$ is precisely the collection of algebras $C(S^3_{00\theta})$, and hence precisely the
collection of algebras $C(S^3_{pq\theta})$ by \cite[Theorem~2.8]{BaumHajacEtAl:K-th05}.
\end{example}

\appendix
\section{Connections with cubical homology}\label{app:grandis}

In this section we show that each $k$-graph determines a cubical set $\widetilde{Q}(\Lambda)$ and that our homology is
isomorphic to that of $\widetilde{Q}(\Lambda)$ as defined by Grandis \cite{Grandis2005}. To define
$\widetilde{Q}(\Lambda)$ we must make sense of degeneracy maps and degenerate cubes in a $k$-graph (see
Definition~\ref{def:cubical} below), and avoiding this was one motivation for providing a self-contained approach in
Section~\ref{sec:homology} above. We could instead have made use of Khusainov's approach \cite{Khusainov:SMZ08} using
semicubical sets. This is in a sense more natural for $k$-graphs since it does not involve degeneracies: it is straightforward to show
that the collection $Q_*(\Lambda)$ of cubes in a $k$-graph forms a semicubical set. However, the sign convention for the
boundary maps in Khusainov's definition of homology differs from those of both Grandis and Massey \cite{Massey}.

Recall the following definition adapted from \cite[\S 1.2]{Grandis2005}. In order to avoid a clash of notation we use $f_i$ for
the degeneracy maps; we also use $1,0$ in place of $+,-$.

\begin{dfn} \label{def:cubical}
A \emph{cubical set}  is a triple $X = (X_r ,\partial_i^\ell , f_i )$ consisting of a sequence $(X_r)_{r = 0}^\infty$ of sets,
together with, for each $r \in \NN$, maps
\[
\partial_{i}^\ell : X_r \to X_{r-1}\quad  
\text{ $\ell \in \{0,1\}$, $1 \le i \le r$}
\quad\text{ and }\quad  f_{i} : X_{r-1} \to X_r \quad\text{$1 \le i \le r$} 
\]
satisfying the \textit{cubical relations}
\begin{align}
\partial_i^\ell \partial_j^m &= \partial_j^m\partial_{i+1}^\ell  \text{ if } j \le i  , \label{eq:facerels} \\
f_i f_j &=  f_{i+1} f_j \ \text{ if } j \le i  , \label{eq:degenrels}
\\
\partial_i^\ell f_j &=
\begin{cases}
f_j \partial_{i-1}^\ell & \text{ if } j < i  ,\\
\operatorname{id} & \text{ if } j=i , \\
f_{j-1} \partial_i^\ell & \text{ if } j > i \label{eq:interels} .
\end{cases}
\end{align}
The maps $\partial^\ell_{i}$ are called \textit{faces} and the $f_{i}$ are called \textit{degeneracies}.
\end{dfn}

We now introduce the $k$-graph analog $\Cube{}$ of the model cocubical set $\mathbb{I}$ described in \cite[\S
1.2]{Grandis2005} (that is, an object satisfying conditions dual to those set out in Definition~\ref{def:cubical}).   Recall from Section~\ref{sec:prelims} that
 for $r \ge 1$, $\1_r= \sum_{i=1}^r e_i$ (and $\1_0 := 0 \in \NN^0$).
We define (see Examples \ref{ex:kgraphs}).
\[
\Cube{r} =
\begin{cases}
\Omega_{r, \mathbf{1_r}}& \textrm{if }\ r \ge 1; \\ 
\Omega_{0} &  \textrm{if }\ r  =  0.
\end{cases}
\]

For $\ell = 0,1$ define $\varepsilon^\ell _0 : \NN^0 \to \NN^1$ by $\varepsilon^\ell _0 (0) = \ell$. For
$1 \le i \le r+1$ and $\ell  \in \{0,1\}$ define $\varepsilon^\ell _i : \NN^{r}  \to \NN^{r+1}$  by
\[
\varepsilon^\ell _i(n_1, \dots, n_r) =
    (n_1, \dots, n_{i-1}, \ell  , n_i \dots, n_r).
\]

If $m \le n \le \mathbf{1}_r$ in $\NN^{r}$, then $\varepsilon^\ell _i (m) \le \varepsilon^\ell _i (n) \le \mathbf{1}_{r+1}$
in $\NN^{r+1}$; so we may extend $\varepsilon^\ell _i$ to a quasimorphism from $\Cube{r}$ to $\Cube{r+1}$ by setting
$\varepsilon^\ell _i(m,n) := (\varepsilon^\ell _i(m), \varepsilon^\ell_i(n))$.

Define $\eta_1 : \NN^1 \to \NN^0$ by $\eta_1 (n) = 0$ for all $n \in \NN$. For $r \ge 2$ and $1 \le i \le r$ we define
$\eta_i : \NN^r \to \NN^{r-1}$ by deleting the $i^\text{th}$ coordinate:
\[
\eta_i (n_1, \dots, n_r) := (n_1, \dots n_{i-1} , n_{i+1}, \dots n_r).
\]

If $m \le n \le \mathbf{1}_r$ in $\NN^{r}$, then $\eta_i (m) \le
\eta_i (n) \le \mathbf{1}_{r-1}$ in $\NN^{r-1}$; so $\eta_i$ extends to a quasimorphism from
$\Cube{r}$ to $\Cube{r-1}$ such that $\eta_i (m,n) = ( \eta_i (m) ,  \eta_i (n) )$.

\begin{prop}  \label{prop:modeliscubical}
The collection $\Cube{} = ( \Cube{n} , \varepsilon^\ell _i , \eta_i )$ forms a cocubical set.
\end{prop}
\begin{proof}
It is routine but tedious to check that the duals of the relations \eqref{eq:facerels}, \eqref{eq:degenrels} and \eqref{eq:interels}
hold.
\end{proof}

Now we build a cubical set $\widetilde{Q} ( \Lambda )$ from a $k$-graph $\Lambda$ by considering collections of maps from
$\Cube{}$ into $\Lambda$: Given $t,r,k \in \NN$, a homomorphism $h : \NN^r \to \NN^k$ is called an \emph{admissible
map of rank $t$}, or just an \emph{admissible map}, if there exist $1 \le i_1 < \cdots < i_t \le r$ and $1 \le j_1 <
\cdots < j_t \le k$ such that
\begin{equation} \label{eq:admiss}
h(e_{i_p}) = e_{j_p}\text{ for $p \le t$}\qquad\text{ and }\qquad
h(e_i) = 0\text{ if $i \not\in \{i_1, \dots, i_t\}$.}
\end{equation}

Let $\Lambda$ be a $k$-graph and fix $r \in \NN$. A quasimorphism $\varphi : \Cube{r} \to \Lambda$ is said to be an
\emph{$r$-cube} if there is an admissible map $h : \NN^r \to \NN^k$ such that $d_\Lambda \circ \varphi = h \circ
d_{\Cube{r}}$.  We say that an $r$-cube $\varphi$ has rank $t$ if the associated admissible map has rank $t$. For $r \ge 0$ let
\[
\widetilde{Q}_r(\Lambda) = \{\varphi : \Cube{r} \to \Lambda : \text{ $\varphi$ is an $r$-cube} \}.
\]
For $1 \le i \le r+1$ and $\ell \in \{0,1\}$, define $\overline{\varepsilon}^\ell _i :
\widetilde{Q}_{r+1}(\Lambda) \to \widetilde{Q}_r(\Lambda)$ by
\begin{align*}
\overline{\varepsilon}^\ell_i(\varphi) &:= \varphi \circ \varepsilon^\ell _i
\intertext{and for $1 \le i \le r$, define
$\overline{\eta}_i : \widetilde{Q}_{r-1}(\Lambda) \to \widetilde{Q}_r (\Lambda)$
by}
 \overline{\eta}_i (\varphi) &:= \varphi \circ \eta_i.
\end{align*}
\begin{rmk}\label{rmk:admissible maps}
Let $\varphi$ be an $(r+1)$-cube of rank $t$ with admissible map $h : \NN^{r+1} \to \NN^k$ given as in
equation~\eqref{eq:admiss} above. If $j = i_p$ for some $p$, then $\overline{\varepsilon}^\ell _j(\varphi)$ is an $r$-cube
whose rank is $t-1$. Otherwise it is an $r$-cube of rank $t$. In either case, the associated admissible map $h' : \NN^r \to
\NN^k$ is given by
\begin{equation}  \label{eq:face admiss}
h'(e_i) = \begin{cases}
        e_{j_p} &\text{ if $i < j$ and $i  = i_p$ for some $p$} \\
        e_{j_p} &\text{ if $i \ge j$ and $i  = i_p - 1$ for some $p$} \\
        0 &\text{ otherwise.}
    \end{cases}
\end{equation}
So $h' ( t_1 , \ldots , t_{r-1} ) = h ( t_1 , \ldots, t_{j-1}, 0,  t_j, \ldots , t_{r-1})$.

Similarly, if $\varphi$ is an $r$-cube of rank $t$ with admissible map $h : \NN^r \to \NN^k$ given
in equation~\eqref{eq:admiss} above, then $\overline{\eta}_j(\varphi)$ is an $(r+1)$-cube of rank
$t$ whose admissible map is given by
 \begin{equation}  \label{eq:degeneracy admiss}
h''(e_i) = \begin{cases}
        e_{j_p} &\text{ if $i < j$ and $i  = i_p$ for some $p$} \\
        e_{j_p} &\text{ if $i > j$ and $i  = i_p + 1$ for some $p$} \\
        0 &\text{ otherwise.}
    \end{cases}
\end{equation}
So $h''( t_1 , \ldots , t_{r+1} ) = h ( t_1 , \ldots, t_{j-1}, t_{j+1} , \ldots , t_{r+1})$.
\end{rmk}

\begin{thm} \label{thm:lambdaiscubical}
Let $\Lambda$ be a $k$-graph. Then $\widetilde{Q} ( \Lambda ) = \big( \widetilde{Q}_r ( \Lambda \big)  ,
\overline{\varepsilon}^\ell _i , \overline{\eta}_i )$ is a cubical set.
\end{thm}

\begin{proof}
This follows from Proposition \ref{prop:modeliscubical}.
\end{proof}

In \cite[\S 2.1]{Grandis2005} the homology of a cubical set is defined as follows: Let $X = ( X_r,
\partial^\ell_i , f_i )$ be a cubical set, then for $n \ge 1$ we define
\[\textstyle
\operatorname{Deg}_r (X) = \bigcup_{i=1}^r \image ( f_i : X_{r-1} \to X_r ) \subseteq X_r
\]
and set $\operatorname{Deg}_0 (X) = \emptyset$. The (normalised) chain complex $(C_* (X) , \partial_*)$ is defined by
\begin{align*}
C_r (X) &= \ZZ X_r / \ZZ \operatorname{Deg}_r (X) = \ZZ \overline{X}_r \text{ where } \overline{X}_r
                = X_r \backslash \operatorname{Deg}_r (X)  \\
\partial_r ( x ) &= \sum_{i,\ell} (-1)^{i+\ell} \,\partial^\ell_i x  \quad \text{ where }x \in \overline{X}_r.
\end{align*}
The homology of $X$ is then the homology of the complex $(C_* ( X) , \partial_*)$, so that
\[
H_r (X) = \ker \partial_r / \image \partial_{r+1} .
\]

An $r$-cube $\varphi : \Cube{r} \to \Lambda$ is called \emph{degenerate} if its rank is strictly less than $r$. Otherwise it is
said to be \emph{nondegenerate}. We define
\begin{align*}
\overline{Q}_r ( \Lambda )  &= \{\varphi : \Cube{r} \to \Lambda : \text{ $\varphi$ is a nondegenerate $r$-cube}\} \\
D_r ( \Lambda ) &= \{\varphi : \Cube{r} \to \Lambda : \text{ $\varphi$ is a degenerate $r$-cube}\},
\end{align*}
so $\widetilde{Q}_r(\Lambda) = \overline{Q}_r (\Lambda) \sqcup D_r(\Lambda)$.

\begin{lem} \label{lem:degen}
Let $\Lambda$ be a $k$-graph. Then
\begin{enumerate}
\item for $1 \le i \le r$ and $\ell = 0,1$, $\overline{\varepsilon}^\ell _i : \widetilde{Q}_{r+1} ( \Lambda ) \to
    \widetilde{Q}_{r} ( \Lambda )$ preserves nondegenerate cubes, that is for  $\varphi \in \overline{Q}_{r+1} (\Lambda )$
    we have $\overline{\varepsilon}^\ell _i ( \varphi ) \in \overline{Q}_{r} ( \Lambda)$;
\item for $1 \le i \le r$ and any $\varphi \in \widetilde{Q}_{r-1} ( \Lambda )$ we have
    $\overline{\eta}_i ( \varphi ) \in D_r ( \Lambda )$;
\item for all $r \ge 1$ we have $D_r ( \Lambda ) = \bigcup^{r}_{i=1}
    \overline{\eta}_i \big(\widetilde{Q}_{r-1} ( \Lambda )\big)$.
\end{enumerate}
\end{lem}

\begin{proof}
For (1), suppose that $\varphi : \Cube{r+1} \to \Lambda$ has rank $r+1$.  Then $\overline{\varepsilon}^\ell _i (\varphi) :
\Cube{r} \to \Lambda$ has rank $r$; so $\overline{\varepsilon}^\ell _i (\varphi) \in \overline{Q}_{r} ( \Lambda)$.

For (2), suppose that $\varphi : \Cube{r-1} \to \Lambda $ has rank $t \le r-1$.  Then $\overline{\eta}_i (\varphi) : \Cube{r}
\to \Lambda$ has rank $t < r$; so $\overline{\eta}_i (\varphi) \in D_r (\Lambda)$.

For (3), suppose that $\varphi \in D_r ( \Lambda )$, that is $\varphi : \Cube{r} \to \Lambda$ has
rank $t < r$. Then there is an admissible map  $h : \NN^r \to \NN^k$ of rank $t$  such that
$d_\Lambda \circ \varphi = h \circ d_{\Cube{r}}$. Let $1 \le i \le r$ be such that $h ( e_i ) = 0$.
Since $\varphi$ does not depend on the $i^\text{th}$ coordinate, we have
$\varphi = \overline{\eta}_i \overline{\varepsilon}_i^0 ( \varphi )$;
hence, $\varphi =  \overline{\eta}_i (\varphi')$ where
$\varphi' = \overline{\varepsilon}_i^0(\varphi)  \in \widetilde{Q}_{r-1}(\Lambda)$.
\end{proof}

Grandis builds his directed homology from the complex given in the following lemma (see \cite[\S 2.1]{Grandis2005}).

\begin{lem}
Let $\Lambda$ be a $k$-graph. Let
\begin{align}
\overline{C}_r ( \Lambda ) &= \ZZ \overline{Q}_r (\Lambda) \nonumber \\
\overline{\partial}_r ( \lambda ) &= \sum_{\ell = 0}^1 \sum_{i=1}^r
(-1)^{i+\ell} \, \overline{\varepsilon}^\ell _i ( \lambda )
\quad \lambda \in \overline{Q}_r ( \Lambda ) \label{eq:newboundary}
\end{align}
Then $(\overline{C} (\Lambda)_* , \overline{\partial}_*)$ is a chain complex.
\end{lem}
\begin{proof}
Theorem~\ref{thm:lambdaiscubical} implies that $\widetilde{Q} ( \Lambda ) = ( \widetilde{Q}_r (\Lambda ) ,
\overline{\varepsilon}^\ell _i , \overline{\eta}_i )$ is a cubical set. By Lemma~\ref{lem:degen} (1) we see that
$\overline{\varepsilon}^\ell _i\big(\overline{Q}_r(\Lambda)\big) \subset \overline{Q}_{r-1} (\Lambda)$ and so
$\overline{\partial}_r$ is well defined. That $\overline{\partial}_{r} \circ \overline{\partial}_{r+1} = 0$ follows from the
property \eqref{eq:facerels} of $\overline{\varepsilon}^\ell _i$. Hence,  $( \overline{C}_* (\Lambda) , \overline{\partial}_*)$
is a complex.
\end{proof}

Our aim is to show that the homology $\overline{H}_* ( \Lambda )$ defined by the complex $(\overline{C}_* ( \Lambda),
\overline{\partial}_*)$ is the same as the homology of the complex $( C_* ( \Lambda ) , \partial_* )$ described in \S 1. We do
this in Theorem~\ref{thm:complexiso} by showing that the complexes are isomorphic. Recall the definition of $Q_r(\Lambda)$
given in~\S2:
\[
Q_r (\Lambda) = \{\lambda \in \Lambda : d (\lambda) \le \mathbf{1}_k , |d(\lambda)| = r\} .
\]

\begin{lem} \label{lem:qbij}
Let $\Lambda$ be a $k$-graph. For $r \ge 0$ and $\lambda \in Q_r (\Lambda)$ there is a unique $\varphi_\lambda \in
\overline{Q}_r (\Lambda)$ such that $\varphi_\lambda(0,\1_{r}) = \lambda$. Conversely, given $\varphi \in \overline{Q}_r
(\Lambda)$, the path $\lambda = \varphi(0,\1_{r}) \in Q_r(\Lambda)$ satisfies  $\varphi_{\lambda} = \varphi$. The map
$\lambda \mapsto \varphi_\lambda$ is a bijection from $Q_r ( \Lambda )$ to $\overline{Q}_r(\Lambda)$ with inverse
$\varphi \mapsto\varphi(0,\1_{r})$.
\end{lem}
\begin{proof}
The result is trivial when $r = 0$ because $\Cube{0} = \{0\}$.

Fix $r \ge 1$ and $\lambda \in Q_r ( \Lambda )$. Let $d ( \lambda ) = e_{i_1} + \cdots + e_{i_r}$, and define an admissible
map $h : \NN^r \to \NN^k$  by $h ( e_j ) = e_{i_j}$ for $j=1, \ldots, r$. Define $\varphi_\lambda : \Cube{r} \to
\Lambda$ by
\[
\varphi_\lambda (m, n) = \lambda(h(m) , h(n)) 
\]
Then $\varphi_\lambda : \Cube{r} \to \Lambda$ is a nondegenerate $r$-cube with $\varphi_\lambda(0,\1_{r}) = \lambda$.
The factorisation property ensures that there is only one nondegenerate cube with range $\lambda$.

Now fix $\varphi \in \overline{Q}_r ( \Lambda )$. Suppose that $d ( \varphi (0,\1_{r}) ) = e_{i_1} + \cdots + e_{i_r}$
with $1 \le i_1 < \cdots < i_r  \le k$.  Let $\lambda = \varphi(0,\1_{r})$ and define $h : \NN^r \to \NN^k$ by $h (e_j) =
e_{i_j}$. Then for $(m,n) \in \Cube{r}$ we have
\[
\varphi_{\lambda} (m,n) = \lambda( h( m ) , h( n ) ) = \varphi ( m , n );
\]
so $\varphi_{\lambda} = \varphi$ as required.
\end{proof}

Recall from Section~\ref{sec:prelims} that for $\lambda \in Q_r ( \Lambda )$, if we express
$d(\lambda)  = e_{i_1} + \cdots + e_{i_r}$ with $1 \le i_1 < \cdots < i_r  \le k$, then
\[
F^0_j(\lambda) = \lambda (0 , d (\lambda) - e_{i_j}) \text{ and } F^1_j (\lambda) = \lambda ( e_{i_j}, d(\lambda)).
\]

\begin{lem} \label{lem:varepisflambda}
Let $\Lambda$ be a $k$-graph and $r \ge 1$. Then for $\lambda \in Q_r ( \Lambda )$ we have
\begin{equation} \label{eq:partialisF}
\overline{\varepsilon}^\ell _j ( \varphi_\lambda ) (0, \mathbf{1}_{r-1}) = F^\ell _j ( \lambda )
\text{ in } Q_{r-1} ( \Lambda ) .
\end{equation}
\end{lem}
\begin{proof}
Let $d ( \lambda ) = e_{i_1} + \cdots + e_{i_r}$ and define $h : \NN^r \to \NN^k$ by $h (e_j) = e_{i_j}$ for $j \le r$.
Then
\begin{flalign*} &&\overline{\varepsilon}^\ell _j ( \varphi_\lambda ) (0, \mathbf{1}_{r-1})
   &=\varphi_\lambda ( \varepsilon^\ell _j (0) , \varepsilon^\ell _j ( \mathbf{1}_{r-1} ) ) & \\
   &&&= \lambda (h( \varepsilon^\ell _j (0)) , h ( \varepsilon^\ell _j (\mathbf{1}_{r-1}) ) ) & \\
   &&&= \begin{cases} \lambda ( 0 , d(\lambda) - e_{i_j} ) & \text{ if } \ell = 0 \\
                      \lambda ( e_{i_j}, d ( \lambda ) ) & \text{ if } \ell  =1
        \end{cases} & \\
   &&&= F^\ell _j ( \lambda ). &\qedhere
\end{flalign*}
\end{proof}

\begin{thm} \label{thm:complexiso}
Let $\Lambda$ be a $k$-graph then the bijection of Lemma~\ref{lem:qbij} induces an isomorphism of complexes
$(C_*(\Lambda) , \partial_*) \cong (\overline{C}_* (\Lambda) , \overline{\partial}_*)$. Hence $\overline{H}_*(\Lambda)
\cong H_*(\Lambda)$.
\end{thm}

\begin{proof}
By Lemma~\ref{lem:qbij} the map $\lambda \mapsto\varphi_\lambda$ induces an isomorphism $\theta_r : C_r ( \Lambda )
\to \overline{C}_r ( \Lambda)$. Let  $\lambda \in Q_r ( \Lambda )$.  By Lemma~\ref{lem:varepisflambda} we have
$\theta_{r-1}(F^\ell _i ( \lambda )) = \overline{\varepsilon}^\ell_i ( \varphi_\lambda )$ for $i = 1, \dots , r$ and $\ell = 0,
1$.  Hence, by \eqref{eq:bondarydef} and \eqref{eq:newboundary} we have
\[
\overline{\partial}_r \theta_r(\lambda ) = \theta_{r-1} \partial_r(\lambda)
\]
and the result follows.
\end{proof}

\section{Topological realisations}\label{app:realisations}

Given a $k$-graph $\Lambda$ we show that the topological realisation $X_\Lambda$ of $\Lambda$ is homeomorphic to the
topological realisation $\mathcal{R} \widetilde{Q} ( \Lambda )$ of the associated cubical set $\widetilde{Q} ( \Lambda )$ as
defined in \cite[\S 1.8]{Grandis2005}. We define the cocubical set
 $\bI^* = ( \bI^r, \dot{\varepsilon}^\ell_i , \dot{\eta}_i )$ of \cite{Grandis2005}
as follows  (we modify Grandis' notation to align with ours from Appendix~\ref{app:grandis}).
For $r \ge 0$ let $\bI^r$ be the unit cube in $\RR^r$.  For
$1 \le i \le r+1$ and $\ell \in \{0,1\}$ define the  coface maps
$\dot{\varepsilon}^\ell_i : \bI^r \to \bI^{r+1}$ and  for $1 \le i \le r$
define codegeneracy maps $\dot{\eta}_i : \bI^r \to \bI^{r-1}$  by
\[
\dot{\varepsilon}^\ell_i(t)_j
    = \begin{cases}
        t_j &\text{ if $j < i$}\\
        \ell &\text{ if $j = i$}\\
        t_{j-1} &\text{ if $j > i$}
    \end{cases}
    \qquad\text{ and }\qquad
\dot{\eta}_i(t)_j
    =\begin{cases}
        t_j &\text{if $j < i$} \\
        t_{j+1} &\text{ if $j \ge i$.}
    \end{cases}.
\]

Recall from \cite{Grandis2005} that $\mathcal{R}\widetilde{Q}(\Lambda)$ is a topological space
endowed with maps $\widehat{\varphi} : \bI^r \to \mathcal{R}\widetilde{Q}(\Lambda)$ for each
$\varphi \in \widetilde{Q}_r(\Lambda)$ satisfying
\begin{equation}\label{eq:grandis hat identities}
    \widehat\varphi \circ \dot{\varepsilon}_i^\ell =
    (\overline{\varepsilon}^\ell_i (\varphi))\;\widehat{\,} \qquad\text{ and }\qquad
    \widehat\varphi \circ \dot{\eta}_i = (\overline{\eta}_i (\varphi))\;\widehat{\,},
\end{equation}
and is uniquely determined by the property that for any topological space $X$ and any collection of
continuous maps $\{\widetilde{\varphi} : \bI^r \to X \mid 1 \le r, \varphi \in
\widetilde{Q}_r(\Lambda)\}$ satisfying
\begin{equation}\label{eq:grandis tilde identities}
    \widetilde\varphi \circ \dot{\varepsilon}_i^\ell =
    (\overline{\varepsilon}^\ell_i (\varphi))^\sim \qquad\text{ and }\qquad
    \widetilde\varphi \circ \dot{\eta}_i = (\overline{\eta}_i (\varphi))^\sim,
\end{equation}
there is a unique continuous map $\pi : \mathcal{R}\widetilde{Q}(\Lambda) \to X$ satisfying $\pi
\circ \widehat{\varphi} = \widetilde{\varphi}$ for all $\varphi \in \widetilde{Q}(\Lambda)$.

Fix $\varphi \in \widetilde{Q}_r (\Lambda )$ and let $h : \NN^r \to \NN^k$ be the associated admissible map. As in
\cite{kkqs} extend $h$ to a map from $\RR^r$ to $\RR^k$ by setting $h(t) := \sum^r_{i=1} t_i h(e_i)$. We define a map
$\widetilde\varphi : \bI_r \to X_\Lambda$ by
\begin{equation}\label{eq:tilde phi}
    \widetilde\varphi(t) = [\varphi(0,\1_r),h(t)].
\end{equation}

\begin{lem}\label{lem:tildephis}
Let $\Lambda$ be a $k$-graph. The maps $\widetilde{\varphi} : \bI_r \to X_\Lambda$ of~\eqref{eq:tilde phi} are
continuous, and satisfy~\eqref{eq:grandis tilde identities}.
\end{lem}
\begin{proof}
Fix $\varphi \in \widetilde{Q}_r(\Lambda)$ with associated admissible map $h$. Since $t \mapsto (\varphi(0, \1_{r}), h(t))$
is continuous from $\bI^r$ to $\{\varphi(0, \1_{r})\} \times [0, h(\1_r)]$, and since the quotient map from
$\bigsqcup_{\lambda \in Q(\Lambda)} \{\lambda\} \times [0, d(\lambda)]$ to $X_\Lambda$ is also continuous, the map
$\widetilde{\varphi}$ is continuous. We check the identities~\eqref{eq:grandis tilde identities}. The calculations are routine but
tedious so we only give a detailed proof of the first identity $\widetilde\varphi \circ \dot{\varepsilon}_i^\ell =
(\overline{\varepsilon}^\ell_i (\varphi))^\sim$, this being the more complicated of the two calculations. The second identity
follows from similar calculations. Define $h': \RR^{r-1} \to \RR^k$
as in Remark~\ref{rmk:admissible maps} by $h' ( t_1 , \ldots , t_{r-1} )
= h ( t_1 , \ldots, t_{i-1}, 0, t_i , \ldots , t_{r-1})$. For $t \in \bI_{r-1}$
\begin{align}
(\widetilde\varphi \circ \dot{\varepsilon}_i^\ell)(t)
    =\widetilde\varphi(\dot{\varepsilon}_i^\ell(t))
    =[\varphi(0,\1_{r}), h(\dot{\varepsilon}_i^\ell(t))] 
    =[\varphi(0,\1_{r}), h'(t) + \ell h(e_i)].\label{eq:compatibility 1}
\end{align}
Since $h'$ is the admissible map associated to $\overline{\varepsilon}_i^\ell(\varphi)$, we also have
\begin{align}
(\overline{\varepsilon}_i^\ell(\varphi))^\sim(t)
    &= [\overline{\varepsilon}_i^\ell(\varphi)(0,\1_{r-1}), h'(t)]
    = [\varphi(\varepsilon_i^\ell (0 , \1_{r-1})), h'(t)]
\label{eq:compatibility 2}
\end{align}
Since $\ell$ is an integer, $h'(t) + \ell h(e_i) - \lfloor h'(t) + \ell h(e_i) \rfloor = h'(t) -
\lfloor h'(t) \rfloor$. Moreover, by the factorisation property, we have
\[
\varphi(0,\1_{r})
    = \varphi(0, \varepsilon_i^\ell(0))\varphi(\varepsilon_i^\ell(0 , \1_{r-1})) \varphi(\varepsilon_i^\ell(\1_{r-1}), \1_r).
\]
Hence, considering separately the cases $\ell = 0$ and $\ell = 1$, one can verify that
\[
\varphi(0,\1_{r})(\lfloor h'(t) + \ell h(e_i) \rfloor, \lceil h'(t) + \ell h(e_i) \rceil)
    = \varphi(\varepsilon_i^\ell(0 , \1_{r-1}))(\lfloor h'(t)\rfloor, \lceil h'(t) \rceil)).
\]
The definition~\eqref{eq:sim def} of the equivalence relation $\sim$ then gives
\[
(\varphi(0,\1_{r}), h'(t) + \ell h(e_i))
    \sim (\varphi(\varepsilon_i^\ell(0 , \1_{r-1})), h'(t)).
\]
Combining this with \eqref{eq:compatibility 1}~and~\eqref{eq:compatibility 2} establishes the first identity
in~\eqref{eq:grandis tilde identities}.
\end{proof}

By Lemma~\ref{lem:tildephis} and the defining property of $\mathcal{R}\widetilde{Q}(\Lambda)$,
there is a unique continuous map $\pi : \mathcal{R}\widetilde{Q}(\Lambda) \to X_\Lambda$ such that
$\pi \circ \widehat{\varphi} = \widetilde{\varphi}$ for all $\varphi \in \widetilde{Q}(\Lambda)$.

\begin{thm}
Let $\Lambda$ be a $k$-graph. The map $\pi:\mathcal{R}\widetilde{Q}(\Lambda) \to X_\Lambda$ is a homeomorphism.
\end{thm}
\begin{proof}
We construct a continuous inverse $\psi$ for $\pi$. Define $\psi_0 : \bigsqcup_{d(\lambda) \le
\1_k} \{\lambda\} \times [0, d(\lambda)] \to \mathcal{R}\widetilde{Q} (\Lambda )$ by
\[
    \psi_0(\lambda, t) := \widehat{\varphi}_\lambda(t),
\]
where $\varphi_\lambda : \Cube{|\lambda|} \to \Lambda$ is the $k$-graph quasimorphism canonically
associated to $\lambda$. The map $\psi_0$ is clearly continuous.

If $\psi([\mu,s]) := \psi_0(\mu,s)$ determines a well-defined map $\psi : X_\Lambda \to
\mathcal{R}\widetilde{Q}(\Lambda)$, then it will be continuous by definition of the topology on $X_\Lambda$, and will be
an inverse for $\pi$. So suppose that $(\mu,s) \sim (\nu,t)$ where $\mu,\nu \in Q(\Lambda)$. Let $I_{(\mu,s)} := \{j :
d(\mu)_j = 1\text{ and }s_j \in \{0,1\}\}$, and define $I_{(\mu,t)}$ similarly. List $I_{(\mu,s)} = \{j_1, \dots, j_p\}$
where $j_1  < \cdots < j_p$. Define $F_{(\mu, s)}$ to be the composition of face maps $F_{(\mu, s)} =
F^{s_{j_1}}_{j_1} \circ \cdots \circ F^{s_{j_p}}_{j_p}$ (with the convention that if $I_{(\mu,s)} = \emptyset$, then
$F_{(\mu,s)}$ is the identity map), and define $F_{(\nu,t)}$ similarly. Then
\[
F_{(\mu,s)}(\mu) = \mu(\lfloor s \rfloor, \lceil s \rceil)
    =  \nu(\lfloor t \rfloor, \lceil t \rceil) = F_{(\nu,t)}(\nu)
\]
because $[\mu,s] = [\nu,t]$. Let $s' := s - \lfloor s \rfloor$ and $t' := t - \lfloor t \rfloor$.
Then
\[
(\mu,s) \sim (F_{(\mu,s)}(\mu), s') = (F_{(\nu,t)}(\nu), t') \sim (\nu,t),
\]
so it suffices to show that $\psi_0(\mu,s) = \psi_0(F_{(\mu,s)}(\mu), s')$. Let
$\dot{\varepsilon}_{(\mu,s)} : \bI_{|\mu| - |I_{(\mu,s)}|} \to \bI_{|\mu|}$ be the composition
$\dot{\varepsilon}^{s_{j_p}}_{j_p} \circ \cdots \circ \dot{\varepsilon}^{s_{j_1}}_{j_1}$. Let
$\overline{\varepsilon}_{(\mu,s)}$ be the composition of face maps in $\widetilde{Q} ( \Lambda )$
corresponding to $F_{(\mu,s)}$. It is routine to see that
\[
\varphi_{F_{(\mu,s)}(\mu)} = \overline{\varepsilon}_{(\mu,s)}(\varphi_\mu).
\]
Hence the identities~\eqref{eq:grandis hat identities} imply that
\[
\widehat{\varphi}_{F_{(\mu,s)}(\mu)}
    = (\overline{\varepsilon}_{(\mu,s)}(\varphi_\mu))\ \widehat{}
    = \widehat{\varphi}_\mu \circ \dot{\varepsilon}_{(\mu,s)}.
\]
In particular, if $\overline{s}$ and $\overline{s}'$ are the elements of $\bI_{|\mu|}$ and
$\bI_{|\mu| - |I_{(\mu,s)}|}$ which map to $s$ and $s'$ under the associated admissible maps, then
\[
\psi_0(F_{(\mu,s)}(\mu), s')
    = \widehat{\varphi}_{F_{(\mu,s)}(\mu)}(\overline{s}')
    = \widehat{\varphi}_\mu \circ \dot{\varepsilon}_{(\mu,s)}(\overline{s}')
    = \widehat{\varphi}_\mu(\overline{s})
    = \psi_0(\mu,s). \qedhere
\]
\end{proof}

\end{document}